\documentclass[11pt, reqno]{amsart}

\usepackage[utf8]{inputenc}
\usepackage[british]{babel}
\usepackage[a4paper, margin=2.5cm]{geometry}
\usepackage{amsmath, amsthm, amssymb, amsfonts}
\usepackage{mathtools}
\usepackage{thmtools}
\usepackage[shortlabels]{enumitem}
\usepackage[font={small, it}]{caption}
\usepackage{subcaption}
\usepackage{microtype}
\usepackage{xcolor}
\usepackage{mathabx}
\usepackage[backref]{hyperref}
\usepackage{tikz}

\usetikzlibrary{calc,decorations.pathreplacing,decorations.pathmorphing,decorations.markings,shapes}

\newcommand{\abspic}[2] {
  \tikzset{n1/.style={thick,draw, inner sep=0pt, minimum size=5pt,fill=black,circle}}
  \tikzset{n2/.style={thick,draw, inner sep=0pt, minimum size=5pt,fill=white,circle}}
  \node[n1] (x) at (1,0) {};
  \node[n2] (y) at (2*#2+1,1) {};
  \foreach \i in {1,...,#1} {
    \node[n2] (a\i) at (2*\i-1,1) {};
    \node[n1] (b\i) at (2*\i,1) {};
  }
  \foreach \i in {1,...,#2} {
    \node[n2] (c\i) at (2*\i,0) {};
    \node[n1] (d\i) at (2*\i+1,0) {};
  }
}

\newcommand{\cH}{\ensuremath{\mathcal H}}
\newcommand{\cI}{\ensuremath{\mathcal I}}
\newcommand{\cJ}{\ensuremath{\mathcal J}}

\newcommand{\cP}{\ensuremath{\mathcal P}}

\newcommand{\eps}{\varepsilon}
\renewcommand{\phi}{\varphi}
\renewcommand{\rho}{\varrho}

\DeclareMathOperator*{\N}{\mathbb{N}}

\DeclareMathOperator*{\R}{\mathbb{R}}

\DeclarePairedDelimiter\ceil{\lceil}{\rceil}
\DeclarePairedDelimiter\floor{\lfloor}{\rfloor}

\let\setminus=\smallsetminus

\newcommand{\osref}[2]{%
  \setlength\abovedisplayskip{5pt plus 2pt minus 2pt}
  \setlength\abovedisplayshortskip{5pt plus 2pt minus 2pt}
  \ensuremath{\overset{\text{#1}}{#2}}
}

\newcommand{\Gnp}{G_{n, p}}

\declaretheorem[parent=section]{theorem}
\declaretheorem[sibling=theorem]{lemma}

\declaretheorem[sibling=theorem]{claim}

\declaretheorem[sibling=theorem,style=definition]{definition}

\setlist{itemsep=0.1em, topsep=0.1em, parsep=0.1em, partopsep=0.1em}

\newcounter{sarrow}

\hypersetup{
	colorlinks,
	linkcolor={red!60!black},
	citecolor={green!50!black},
	urlcolor={blue!80!black}
}

\colorlet{RoyalRed}{red!70!black}
\definecolor{RoyalBlue}{rgb}{0.25, 0.41, 0.88}
\definecolor{RoyalAzure}{rgb}{0.0, 0.22, 0.66}

\title{Covering cycles in sparse graphs}

\author{Frank Mousset}
\address{Institute of Mathematical Sciences, Tel Aviv University, Tel Aviv,
Israel}
\email{moussetf@inf.ethz.ch}

\author{Nemanja \v{S}kori\'{c}}
\address{Institute of Theoretical Computer Science, ETH Z\"{u}rich, 8092
Z\"{u}rich, Switzerland}
\email{nskoric@inf.ethz.ch}

\author{Milo\v{s} Truji\'{c}}
\address{Institute of Theoretical Computer Science, ETH Z\"{u}rich, 8092
Z\"{u}rich, Switzerland}
\email{mtrujic@inf.ethz.ch}

\date{\today}

\thanks{Research supported by Israel Science Foundation grants 1028/16 and
1147/14, and ERC Starting Grant 633509 (FM), and by grant no.\ 200021 169242
of the Swiss National Science Foundation (MT)}

\begin{document}

\begin{abstract}
  Let $k \geq 2$ be an integer. Kouider and Lonc proved that the vertex set of
  every graph $G$ with $n \geq n_0(k)$ vertices and minimum degree at least
  $n/k$ can be covered by $k-1$ cycles. Our main result states that for every
  $\alpha > 0$ and $p = p(n) \in (0,1]$, the same conclusion holds for graphs
  $G$ with minimum degree $(1/k+\alpha)np$ that are sparse in the sense that
  \[
    e_G(X,Y) \leq p|X||Y| + o(np\sqrt{|X||Y|}/\log^3 n) \qquad \forall
    X,Y\subseteq V(G).
  \]
  In particular, this allows us to determine the local resilience of random and
  pseudorandom graphs with respect to having a vertex cover by a fixed number of
  cycles. The proof uses a version of the absorbing method in sparse expander
  graphs.
\end{abstract}

\maketitle

\section{Introduction}

A classical result of Dirac states that every graph with $n \geq 3$ vertices and
minimum degree at least $n/2$ contains a Hamilton cycle, that is, a cycle
passing through all vertices of the graph~\cite{dirac1952some}. There exist a
vast number of extensions of this theorem, most of which state that every graph
satisfying a certain minimum degree condition must have some `global' structure.

For example, Hajnal and Szemer\'{e}di~\cite{hajnal1970proof} proved that, for
every $k \geq 2$, the vertex set of every graph with $n$ vertices and minimum
degree at least $(k-1)n/k$ can be covered by vertex-disjoint copies of $K_k$,
provided that $k$ divides $n$. P\'{o}sa~\cite{erdos1964problem} and
Seymour~\cite{seymour1973problem} conjectured that, more generally, every graph
with $n$ vertices and minimum degree at least $(k-1)n/k$ contains the $(k-1)$-st
power of a Hamilton cycle, that is, a Hamilton cycle in which every pair of
vertices at distance at most $k-1$ is connected by an edge (the case $k=2$ is
Dirac's theorem). This conjecture was first proved for large $n$ by Koml\'{o}s,
S\'{a}rk\"{o}zy, and Szemer\'{e}di~\cite{komlos1998proof}, using the regularity
method, and later for smaller values of $n$ by Levitt, S\'{a}rk\"{o}zy, and
Szemer\'{e}di~\cite{levitt2010avoid} and Chau, DeBiasio, and
Kierstead~\cite{chau2011posa}. A minimum degree condition ensuring the presence
of more general subgraphs, formulated in terms of the chromatic number, is given
by the bandwidth theorem of B\"{o}ttcher, Schacht, and
Taraz~\cite{bottcher2009proof}.

The above results all concern graphs with rather large minimum degrees. In the
case where the minimum degree can be smaller than $\floor{n/2}$, the graph might
not be connected, so one can no longer guarantee any spanning connected
subgraph; similarly, the graph might be bipartite, and one cannot guarantee any
non-bipartite subgraph. However, such graphs may still have some interesting
global properties. The following extension of Dirac's theorem to graphs with
minimum degree below $n/2$ was first conjectured by Enomoto, Kaneko, and
Tuza~\cite{enomoto1987p_3} and proved by Kouider and
Lonc~\cite{kouider1996covering}.

\begin{theorem}[\cite{kouider1996covering}]\label{thm:KL}
  Let $k \geq 2$ be an integer and let $G$ be a graph with $n$ vertices and
  minimum degree at least $n/k$. Then the vertex set of $G$ can be covered by $k
  - 1$ cycles, edges, or vertices.
\end{theorem}

We note that if $n=n(k)$ is large enough, then `cycles, edges, or vertices' can
be replaced by `cycles' (this follows for example from the main result
in~\cite{balogh2017stability}).

There is a trend in modern combinatorics towards proving sparse analogues of
extremal results (see, e.g., Conlon's survey~\cite{conlon2014combinatorial}).
Our main result in this paper is a sparse analogue of Theorem~\ref{thm:KL}. We
use the following natural notion of uniformly sparse graphs, which can be seen
as a one-sided version of Thomason's jumbled graphs~\cite{thomason1987pseudo,
thomason1987random}.

\begin{definition}[$(p, \beta)$-sparse]
  A graph $G$ is \emph{$(p, \beta)$-sparse} if for all subsets $X, Y \subseteq
  V(G)$,
  \[
    e_G(X, Y) \leq p|X||Y| + \beta \sqrt{|X||Y|}.
  \]
\end{definition}

It is well known that for every $p=p(n)\leq 0.99$, the Erd\H{o}s--R\'{e}nyi
random graph $G_{n,p}$ is $(p,O(\sqrt{np}))$-sparse w.h.p.\footnote{With high
probability, that is, with probability tending to $1$ as $n \to
\infty$.}~\cite{krivelevich2006pseudo}. With this definition, our main result
reads as follows.

\begin{theorem}\label{thm:main-theorem}
  For every integer $k \geq 2$ and every $\alpha > 0$, there exists a positive
  $\eta(\alpha, k)$ such that the following holds for all sufficiently large
  $n$, all $p \in (0, 1]$, and all $\beta \leq\eta np/\log^3 n$. Let $G$ be a
  $(p,\beta)$-sparse graph with $n$ vertices and minimum degree at least $(1/k +
  \alpha)np$. Then the vertex set of $G$ can be covered by $k-1$ cycles.
\end{theorem}

The minimum degree $(1/k+\alpha)np$ cannot be much improved. Indeed, assume
$\log^6 n/n \ll p \ll 1$ and let $G$ be a graph consisting of $k$
vertex-disjoint copies of $G_{n/k,p}$. Then w.h.p.\ $G$ has minimum degree at
least $(1/k-o(1))np$ and is $(p,\beta)$-sparse with $\beta = O(\sqrt{np}) =
o(np/\log^3 n)$, but cannot be covered by $k-1$ cycles. A very similar
construction shows that the upper bound on $\beta$ in our result is optimal up
to the logarithmic factors. To see this, take any $\log n/n \ll p \ll 1$ and
consider the random graph $G$ given by the union of $k$ vertex-disjoint copies
of $G_{n/k, q}$, where $q = (1 + 2k\alpha)p$. Then $G$ cannot be covered by $k -
1$ cycles but, at the same time, w.h.p.\ it has minimum degree at least $(1/k -
o(1))nq \geq (1/k + \alpha)np$ and is $(q, O(\sqrt{nq}))$-sparse. The latter
means that for all $X,Y\subseteq V(G)$,
\[
  e(X,Y) \leq q|X||Y| + O(\sqrt{n q|X||Y|}) \leq p|X||Y| + O(np\sqrt{|X||Y|}),
\]
using $p|X||Y|\leq np\sqrt{|X||Y|}$, so $G$ is in fact $(p,O(np))$-sparse.

Our main motivation for studying the problem from
Theorem~\ref{thm:main-theorem} is the connection to the local resilience of
sparse random and pseudorandom graphs. The systematic study of this notion was
initiated by Sudakov and Vu~\cite{sudakov2008local} and, since then, the topic
has garnered considerable attention (see, e.g.,~\cite{allen2020bandwidth,
balogh2011local, balogh2012corradi, lee2012dirac, montgomery2020hamiltonicity,
vskoric2018local} and the surveys~\cite{bottcher2017large,
sudakov2017robustness}).

\begin{definition}[Local resilience]
  Let $\cP$ be a monotone\footnote{A graph property $\cP$ is monotone if it is
  preserved under adding edges.} graph property and let $G$ be a graph in $\cP$.
  The \emph{local resilience of $G$ with respect to $\mathcal P$} is defined as
  the maximum $r\in [0, 1]$ such that, for every $H\subseteq G$ satisfying
  $\deg_H (v) < r \deg_G(v)$ for all $v\in V(G)$, we have $G - H \in \cP$.
\end{definition}

For example, Theorem~\ref{thm:KL} implies that the local resilience of $K_n$
with respect to having a vertex-cover by $k-1$ cycles is at least $(k-1)/k -
o(1)$, which is easily seen to be optimal (consider a disjoint union of $k$
cliques of size $n/k$). Since $\Gnp$ is w.h.p.\ $(p, O(\sqrt{np}))$-sparse and
has degrees concentrated around $np$, Theorem~\ref{thm:main-theorem} has the
following consequence for the local resilience of random graphs.

\begin{theorem}\label{thm:main-theorem-gnp}
  Let $k\geq 2$ be an integer and let $p = p(n)$ be such that $p\gg \log^6 n/n$.
  Then the local resilience of $\Gnp$ with respect to having a vertex-cover by
  $k-1$ cycles is w.h.p.\ $(k-1)/k \pm o(1)$.
\end{theorem}

Indeed, it is not difficult to see that $(k-1)/k+o(1)$ is an upper bound
(consider a random partition of the vertex set into $k$ parts). Note that if $p
\ll \log n/n$, then w.h.p.\ $\Gnp$ has an unbounded number of connected
components and thus no vertex-cover by $k-1$ cycles. We believe that with a bit
more care, one could improve the lower bound on $p$ in the above theorem to
$p\gg \log^4 n/n$, using essentially the same proof strategy. However, there is
a natural barrier in our method that prevents us from going down all the way to
$\log n/n$ (which is the threshold for having a vertex-cover by a constant
number of cycles). Doing so would require new ideas and techniques.

Note also that the case $k = 2$ in Theorem~\ref{thm:main-theorem-gnp}
corresponds to the problem of determining the local resilience $\Gnp$ with
respect to Hamiltonicity. This question was resolved independently by
Montgomery~\cite{montgomery2019hamiltonicity} and by Nenadov, Steger, and the
third author~\cite{nenadov2019resilience}, who showed that the local resilience
of $\Gnp$ with respect to Hamiltonicity is w.h.p.\ $1/2 \pm o(1)$ whenever $p =
(\log n + \log\log n + \omega(1))/n$.

Our last result concerns pseudorandom graphs. An $(n,d,\lambda)$-graph is a
$d$-regular graph with $n$ vertices for which all eigenvalues of the adjacency
matrix, with the exception of the largest one, are bounded in absolute value by
$\lambda$. It follows from the expander mixing lemma that every $(n, d,
\lambda)$-graph is $(d/n, \lambda)$-sparse. Thus, Theorem~\ref{thm:main-theorem}
has the following consequence.

\begin{theorem}\label{thm:main-theorem-ndlambda}
  For every integer $k \geq 2$ and every $\alpha > 0$, there exists a positive
  $\eta(\alpha, k)$ such that the following holds for all sufficiently large
  $n$. Let $G$ be an $(n,d,\lambda)$-graph with $\lambda \leq \eta d/\log^3 n$.
  Then the local resilience of $G$ with respect to having a vertex-cover by $k -
  1$ cycles is $(k - 1)/k \pm \alpha$.
\end{theorem}

Again, it is not difficult to see that $(k-1)/k+o(1)$ is an upper bound.
Theorem~\ref{thm:main-theorem-ndlambda} generalises a result of Sudakov and
Vu~\cite{sudakov2008local} stating that every $(n, d, \lambda)$-graph $G$ with
$\lambda \leq d/\log^2 n$ has local resilience at least $1/2-o(1)$ with respect
to Hamiltonicity.

Theorems~\ref{thm:main-theorem-gnp} and \ref{thm:main-theorem-ndlambda} are, to
the best of our knowledge, the first positive results on the local resilience of
(pseudo)random graphs where the local resilience is significantly larger than
$1/2$. We believe the methods introduced in this paper could be used to tackle
different problems allowing such high resilience.

\subsection{Methods and techniques}

The proof of Theorem~\ref{thm:main-theorem} combines several techniques for
embedding large structures into (pseudo)random graphs and their subgraphs. The
focal point is the {\em absorbing method}, first introduced under this name by
R\"{o}dl, Rucinski, and Szemer\'{e}di~\cite{rodl2006dirac}, but already used
implicitly in earlier works of Erd\H{o}s, Gy\'{a}rf\'{a}s, and
Pyber~\cite{erdHos1991vertex} and Krivelevich~\cite{krivelevich1997triangle}.

We now give a simplified account of how this method approaches the problem of
embedding a spanning graph $S$ (in our case, a spanning union of $k-1$ cycles)
into a graph $G$. First, one reserves a small subset $R$ of the vertex set of
$G$, called the \emph{reservoir}; frequently, this is simply a uniformly random
subset of small size. Then one embeds a certain highly structured graph $A$,
called the \emph{absorber}, into $G[V(G) \setminus R]$, such that the following
holds: suppose that we embed a fixed subgraph $S'\subseteq S$ into $G[V(G)
\setminus V(A)]$ in such a way that all vertices outside of $R\cup V(A)$ are
covered; then there exists a completion of the embedding of $S'$ to an embedding
of $S$. We remark that, usually, it is very difficult to control which vertices
of $R$ are used by the embedding of $S'$, so this property relies on a careful
choice of $A$ (we think of $A$ as `absorbing' the vertices in $R$ not covered by
the embedding of $S'$). In this way, the problem of embedding the spanning graph
$S$ is reduced to the (often easier) problem of embedding the non-spanning graph
$S'$. A method similar to the one described in this paragraph has been
successfully applied to numerous problems in (random) graph theory (see,
e.g.,~\cite{ferber2017robust, georgakopoulos2018spanning,
glock2021decompositions, kwan2020almost, levitt2010avoid,
montgomery2019spanning, nenadov2019powers}).

The first step towards carrying out the above approach is to figure out the
structure of the absorber. In our case, one of the main issues to overcome is
that the graph $G$ might be bipartite, which means that, in order to have any
hope of embedding the absorber into $G[V(G) \setminus R]$, the absorber must be
bipartite as well. This creates several technical challenges. Most importantly,
we cannot use the common approach of building an absorber by stringing together
many single-vertex absorbers, as absorbing only a single vertex may upset the
delicate balance between the two parts of the bipartition; rather, we need to be
able to absorb two vertices at a time (one from each part of the bipartition).
This is done using a variation of a trick of
Montgomery~\cite{montgomery2019spanning}. A consequence of this approach is that
our absorber is only able to absorb subsets of $R$ that contain the same number
of vertices from each part of the (hypothetical) bipartition.

Even after deciding upon a suitable structure for the absorber, we still face
the issue of actually embedding this structure into $G[V(G) \setminus R]$. In
the context of sparse pseudorandom graphs, this is usually done by exploiting
the expansion properties of the graph to show that one can connect prescribed
pairs of vertices or edges by disjoint copies of a given fixed graph $F$ (for
example, a path of length $\log n$). This statement is usually referred to as
the \emph{Connecting Lemma}, though its precise formulation depends on the
nature of $S$. The actual absorber is then embedded by multiple uses of this
lemma. A difficulty that arises in our setting (and that is not a problem when
the minimum degree is at least $(1/2+o(1))np$) is that the graph $G$ that we are
dealing with might not be a very good expander at all---in fact, if $k$ is
large, then $G$ might have a large number of connected components, implying that
there are fairly small sets (of size roughly $n/k$) that do not expand at all.
An important step in our proof is to show that $G$ can nevertheless be
partitioned into at most $k - 1$ subgraphs, each having strong enough expansion
properties to embed an absorber, all without sacrificing much of the minimum
degree. We refer to this as the \emph{Partitioning Lemma}.

In order to complete the proof, we cover each of these expanding subgraphs by
systems of disjoint paths, leaving uncovered only vertices in the reservoir and
the absorber. This is done using a standard application of the sparse regularity
lemma in conjunction with the recent bootstrapping argument of Nenadov and the
second author~\cite{nenadov2020komlos}. Finally, the absorber is used to connect
each of these systems of disjoint paths into a cycle (and picking up the
uncovered vertices in the reservoir along the way).

We believe that some of these techniques are likely to be of use for other,
related problems concerning the structure of uniformly sparse graphs satisfying
a minimum degree condition. In particular, the Partitioning Lemma and Connecting
Lemma that we prove are quite general statements that are largely unrelated to
the concrete problem that is solved in this paper.

\subsection{Organisation of the paper}

The paper is structured as follows. In Section~\ref{sec:preliminaries} we
introduce the notion of expander graphs and state several useful properties of
such graphs. Furthermore, we mention some of the more standard tools we use,
namely Haxell's criterion for matchings in hypergraphs and Szemer\'{e}di's
regularity lemma for sparse graphs and related concepts. In
Section~\ref{sec:the-main-proof} we reduce Theorem~\ref{thm:main-theorem} to a
version in which one can assume that the graph is an expander graph. We also
state all the necessary `big gun' lemmas used in order to prove it and
subsequently give a proof of Theorem~\ref{thm:main-expander} modulo those
lemmas. Each of the following
Sections~\ref{sec:inheritance}--\ref{sec:embedding-lemma} are fully dedicated to
the proof of one of the lemmas.

\subsection{Notation}

We use standard graph theoretic notation. In particular, given a graph $G$,
$V(G)$ and $E(G)$ denote the sets of vertices and edges of $G$, respectively. We
write $v(G) = |V(G)|$ and $e(G) = |E(G)|$. For a subset $X \subseteq V(G)$, we
denote by $G[X]$ the subgraph induced by the vertex set $X$. For two (not
necessarily disjoint) subsets $X, Y \subseteq V(G)$, we write $e_G(X, Y)$ for
the number of pairs in $X \times Y$ that form an edge, and $e_G(X)$ for the
number of edges in $X$. Note that $e_G(X, X) = 2e_G(X)$. Furthermore, we denote
by $N_G(X, Y)$ the set of all neighbours in $Y$ of vertices from $X$. We
abbreviate $N_G(\{x\}, Y)$ to $N_G(x, Y)$ and define $\deg_G(x, Y) = |N_G(x,
Y)|$ and $\deg_G(x) = \deg_G(x, V(G))$. We say that a path $P$ \emph{connects}
two vertices $x, y$ if $x$ and $y$ are its endpoints, and say that $P$ is an
$xy$-path. The length of a path is defined as the number of edges in it. For
$\ell \in \N$, we denote by $N_G^{\ell}(X, Y)$ the set of vertices $y \in Y$ for
which there exists a path $P$ of length at most $\ell$ connecting $y$ to some $x
\in X$ and whose internal vertices are in $Y$; in particular $N_G^1(X, Y) =
N_G(X, Y)$. Again, we abbreviate $N_G^\ell(\{x\},Y) = N_G^\ell(x,Y)$. If $X$ and
$Y$ are disjoint, then $G[X,Y]$ is the induced bipartite subgraph with parts $X$
and $Y$, and the {\em density} of the pair $(X, Y)$ is $d_G(X, Y) = e_G(X,
Y)/(|X||Y|)$. In all of the above notations, we may omit the subscript $G$ when
it is clear which graph we are talking about.

For an integer $n$ we write $[n] = \{1, \dotsc, n\}$ and for $a, b, \eps \in
\R$, we write $a \in (1 \pm \eps)b$ to denote $(1 - \eps)b \leq a \leq (1 +
\eps)b$. We use the standard asymptotic notation $o, O, \omega$, and $\Omega$.
The logarithm function is always used with the natural base $e$. We suppress
floors and ceilings whenever they are not crucial. Finally, we use the
convention that if the statement of (say) Lemma 3.6 features a value named $C$,
then we may elsewhere write $C_{3.6}$ to denote this value.

\section{Expansion and other preliminaries}\label{sec:preliminaries}

A simple but important property of $(p,\beta)$-sparse graphs is that every set
$A$ of vertices with degree drastically above $|A|p$ must expand by a
significant amount. This is the content of our first lemma.

\begin{lemma}
  \label{lem:edges-out-small-set}
  Let $p \in [0, 1]$, let $\alpha, \beta > 0$, and let $G$ be a $(p,
  \beta)$-sparse graph on $n$ vertices. Assume $A \subseteq V(G)$ is a subset
  such that $\deg(a) \geq |A|p + \alpha np$ for all $a \in A$. Then
  \[
    e(A, V(G) \setminus A) \geq (\alpha np - \beta) |A|.
  \]
  In particular, if $A$ is not empty, then there exists a vertex $a \in A$ such
  that $\deg(a, V(G) \setminus A) \geq \alpha np - \beta$.
\end{lemma}
\begin{proof}
  Set $B = V(G) \setminus A$. Recalling the definition of $e(\cdot, \cdot)$ our
  assumption implies
  \[
    e(A, B) = e(A, V(G)) - e(A, A) \geq |A|^2p + |A|\cdot \alpha np - e(A, A).
  \]
  On the other hand, as $G$ is $(p, \beta)$-sparse,
  \[
    e(A, A) \leq |A|^2p + \beta|A|.
  \]
  Combining these inequalities gives $e(A, B) \geq (\alpha np - \beta) |A|$, and
  the last assertion follows simply by averaging.
\end{proof}

In particular, if $G$ is a $(p,o(np))$-sparse graph with minimum degree
$\Omega(np)$, then the above lemma shows that all small enough linear-sized
subsets of vertices expand by a factor $\Omega(np)$. An important role in the
proof is played by graphs in which also the larger sets of vertices have this
property. We make the following definition.

\begin{definition}[$q$-expander]
  \label{def:expander}
  Let $q > 0$. A graph $G$ is a \emph{$q$-expander} if for every partition $V(G)
  = V_1 \cup V_2$, we have
  \[
    e(V_1, V_2) \geq q|V_1||V_2|.
  \]
\end{definition}

Informally, we think of a $(p, o(np))$-sparse graph $G$ as being a `good
expander' if it is a $q$-expander for some $q = \Omega(p)$. One can see that
this is essentially best possible, since the definition of a $(p,o(np))$-sparse
graph implies that $e(V_1, V_2) \leq p|V_1||V_2| + \beta \sqrt{|V_1||V_2|} = (1
+ o(1)) p|V_1||V_2|$ for every partition $V = V_1 \cup V_2$ into sets of linear
size.

The following simple lemma allows us to assume without loss of generality that
our expanders are bipartite, which turns out to be convenient later on in the
proof.

\begin{lemma}
  \label{lem:expander-to-bipartite-expander}
  Let $q > 0$. Every $q$-expander $G$ contains a spanning bipartite
  $(q/2)$-expander as a subgraph.
\end{lemma}
\begin{proof}
  Let $V(G) = A \cup B$ be a partition of the vertex set of $G$ that maximises
  $e_G(A, B)$. We claim that $H = G[A, B]$ is a $(q/2)$-expander. Since we
  assume that $G$ is a $q$-expander, it is enough to show that for any partition
  of the vertices into non-empty sets $X$ and $Y$, we have $e_H(X, Y) \geq
  e_G(X, Y)/2$.

  To see this, define the sets $V_1 = X\cap A$, $V_2= X\cap B$, $V_3 = Y\cap A$,
  and $V_4 = Y\cap B$, and note that
  \begin{align*}
      2e_H(X,Y) + e_G(V_1\cup V_4, V_2\cup V_3)
      & =
      e_G(V_1\cup V_2,V_3\cup V_4) + e_G(V_1\cup V_3,V_2\cup V_4) \\
      & = e_G(X,Y) + e_G(A,B);
  \end{align*}
  here, the first equality can be verified by writing $e_H(X,Y) = e_G(V_1,V_4) +
  e_G(V_2,V_3)$ and observing that for all $1\leq i\leq j\leq 4$, both sides of
  the equality count each edge of $G[V_i,V_j]$ the same number of times. The
  maximal choice of $(A, B)$ ensures that $e_G(V_1\cup V_4, V_2\cup V_3) \leq
  e_G(A,B)$. This implies the desired inequality $e_H(X, Y) \geq e_G(X, Y)/2$.
\end{proof}

\subsection{Matchings in hypergraphs}

The following theorem due to Haxell has recently seen a surge of applications in
problems concerning embedding (spanning) structures into sparse graphs. It is
similar to Hall's theorem in spirit, providing a condition for the existence of
a perfect matching in certain hypergraphs.

\begin{theorem}[Haxell's criterion~\cite{haxell1995condition}]
  \label{thm:haxell-matching}
  Let $A$ and $B$ be disjoint sets and let $\cH = (A \cup B, E)$ be an
  $r$-uniform hypergraph such that $|A \cap e| = 1$ and $|B \cap e| = r - 1$ for
  every edge $e \in E$. Suppose that for every choice of subsets $S \subseteq A$
  and $Z \subseteq B$ such that $|Z| \leq (2r-3)(|S|-1)$, there is an edge $e
  \in E$ intersecting $S$ but not $Z$. Then $\cH$ contains an $A$-saturating
  matching (that is, a collection of disjoint hyperedges whose union contains
  $A$).
\end{theorem}

\subsection{Sparse regularity lemma}

Given a graph $G$ and $\eps, p > 0$, we say that a pair $(X, Y)$ of disjoint
subsets $X, Y \subseteq V(G)$ is $(\eps, p)$-{\em regular} if for all subsets
$X' \subseteq X$ and $Y' \subseteq Y$ with $|X'| \geq \eps|X|$ and $|Y'| \geq
\eps|Y|$, we have $|d(X', Y') - d(X, Y)| \leq \eps p$. We say that the pair $(X,
Y)$ is $(\eps, p)$-{\em lower-regular} if for all $X'$ and $Y'$ as above, we
have $d(X', Y') \geq (1 - \eps)p$.

A partition $V(G) = V_0 \cup \dotsb \cup V_t$ is called an $(\eps, p)$-{\em
regular partition with exceptional class $V_0$} if $|V_0| \leq \eps n$, $|V_1| =
\dotsb = |V_t| \leq n/t$, and all but at most $\eps t^2$ pairs $(V_i, V_j)$ with
$1 \leq i < j \leq t$ are $(\eps, p)$-regular.

\begin{lemma}[Sparse regularity lemma~\cite{scott2011szemeredi}]
  \label{lem:sparse-regularity-lemma}
  For all $\eps, m > 0$, there exists $M(\eps, m)$ such that for every graph $G$
  on at least $M$ vertices, there exists an $(\eps, p)$-regular partition
  $(V_i)_{i = 0}^{t}$ of $V(G)$, where $p = e(G)/\binom{n}{2}$ is the density of
  $G$ and $m \leq t \leq M$.
\end{lemma}

Lastly, we need a lemma due to Gerke, Kohayakawa, Rödl, and
Steger~\cite{gerke2007small} stating that in an $(\eps,p)$-lower-regular pair,
almost all subsets of size at least $D/p$, for some $D > 0$, inherit
lower-regularity, with slightly weaker parameters.

\begin{lemma}[Corollary 3.8 in~\cite{gerke2007small}]
  \label{lem:small-subsets-regular}
  For all $\eps', \delta \in (0, 1)$, there exist positive constants
  $\eps_0(\eps', \delta)$ and $D(\eps')$ such that the following holds for all
  $0 < \eps \leq \eps_0$ and $p \in (0, 1)$. Suppose $(V_1, V_2)$ is an $(\eps,
  p)$-lower-regular pair and $q_1, q_2 \geq Dp^{-1}$. Then the number of pairs
  $(Q_1, Q_2)$ with $Q_i \subseteq V_i$ and $|Q_i| = q_i$ ($i = 1, 2$) that are
  $(\eps', p)$-lower-regular is at least
  \[
    (1 - \delta^{\min{\{q_1, q_2\}}}) \binom{|V_1|}{q_1} \binom{|V_2|}{q_2}.
  \]
\end{lemma}

\section{Proof of Theorem~\ref{thm:main-theorem}}\label{sec:the-main-proof}

The proof of Theorem~\ref{thm:main-theorem} involves several different
ingredients. In this section, we gather the necessary definitions and key
lemmas, whose proofs we postpone to the subsequent sections. We then show how
these lemmas can be used to deduce our main result.

Recall, our goal is to show that if $G$ is a $(p, \beta)$-sparse graph with $n$
vertices, minimum degree at least $(1/k + \alpha)np$, and $\beta \leq \eta
np/\log^3 n$, then the vertex set of $G$ can be covered by $k - 1$ cycles. It
turns out that if $G$ is an expander, then it has all kinds of good connectivity
properties that make it easier to cover it by cycles. However, it is not hard to
see that if $k > 2$, then $G$ need not even be connected; in particular, it need
not be a good expander.

The first step of the proof deals with this problem. For a graph $G$ and $\xi
\geq 0$, let us write $\delta_\xi(G)$ for the maximal integer $d$ such that all
but at most $\xi v(G)$ vertices $v$ of $G$ satisfy $\deg_G(v) \geq d$. Note that
$\delta_0(G)$ is simply the minimum degree $\delta(G)$ of $G$. On the other
hand, $\delta_\xi(G)$ for a small constant $\xi > 0$ is a sort of `essential
minimum degree' of $G$ possessed by a $(1-\xi)$-fraction of the vertices.

\begin{restatable}[Partitioning Lemma]{lemma}{partitioning}
  \label{lem:partitioning-lemma}
  For all $c, \alpha, \xi \in (0, 1)$, there exist positive $\gamma(\alpha, \xi,
  c)$ and $\eta(\alpha, \xi, c)$ such that the following holds for all
  sufficiently large $n$. Let $p \in (0, 1)$ and $\beta \leq \eta np$, and let
  $G$ be a $(p, \beta)$-sparse graph on $n$ vertices with minimum degree at
  least $(c + \alpha)np$. Then there exists a partition $V(G) = V_1 \cup \dotsb
  \cup V_\ell$ into $1 \leq \ell < 1/c$ parts such that, for every $i \in
  [\ell]$,
  \begin{enumerate}[(i)]
    \item $G[V_i]$ is a $\gamma p$-expander,
    \item $\delta(G[V_i]) \geq c^2 np$,
    \item $\delta_\xi(G[V_i]) \geq (c + \alpha - \xi)np$.
  \end{enumerate}
\end{restatable}

The proof of this lemma is given in Section \ref{sec:partitioning-lemma}.
Informally, it says that if $G$ has minimum degree slightly above $cnp$, then it
can be partitioned into fewer than $1/c$ good expanders, each of which
\emph{essentially} still has the same minimum degree as $G$ (and each of which
has minimum degree at least $c^2np$). This means that from now on, instead of
working with an arbitrary $(p, \beta)$-sparse graph with minimum degree
$\delta(G) \geq (1/k + \alpha)np$, we can work with a $(p, \beta)$-sparse
\emph{expander graph} satisfying $\delta_\xi (G)\geq (1/k + \alpha - \xi)np$,
where $\xi$ is an arbitrarily small positive constant. Luckily for us, the fact
that we go from a graph with $\delta(G) \geq (c + \alpha)np$ to a graph with
$\delta_\xi(G) \geq (c + \alpha - \xi)np$ does not pose any insurmountable
difficulty. Theorem~\ref{thm:main-theorem} now follows easily from the following
`robust' version specialised to expander graphs.

\begin{theorem}\label{thm:main-expander}
  For every integer $k \geq 2$ and all $\alpha, \gamma \in (0, 1)$, there exists
  a positive $\eta(\alpha, \gamma, k)$, such that the following holds for all
  sufficiently large $n$. Let $p \in (0, 1]$ and $\beta \leq \eta np/\log^3 n$.
  Then every $(p, \beta)$-sparse $\gamma p$-expander $G$ on $n$ vertices
  satisfying
  \[
    \delta(G) \geq 2\alpha np \quad \text{and} \quad \delta_{\alpha/128}(G) \geq
    (1/k + \alpha)np
  \]
  has a vertex cover by $k - 1$ cycles.
\end{theorem}

\begin{proof}[Proof of Theorem~\ref{thm:main-theorem}]
  Without loss of generality we may assume that $\alpha > 0$ is small enough, in
  particular smaller than $1/(2k^2)$. Let $\xi = \alpha/129$, $\alpha' = \alpha
  - \xi$, $\gamma = \gamma_{\ref{lem:partitioning-lemma}}(\alpha, \xi, 1/k)$,
  $\eta' = \min_{2 \leq i \leq k}{\{ \eta_{\ref{thm:main-expander}}(\alpha',
  \gamma, i) \}}$, and $\eta = \min{\{ \alpha\eta'/2,
  \eta_{\ref{lem:partitioning-lemma}}(\alpha, \xi, 1/k) \}}$. Let $G = (V, E)$
  be a $(p, \beta)$-sparse graph with $n$ vertices, minimum degree $(1/k +
  \alpha)np$, and $\beta \leq \eta np/\log^3 n$. We apply
  Lemma~\ref{lem:partitioning-lemma} with $1/k$ (as $c$) to obtain some $1 \leq
  \ell < k$ and a partition $V = V_1 \cup \dotsb \cup V_\ell$ such that, for
  every $i \in [\ell]$,
  \begin{enumerate}[(i), font=\itshape]
    \item $G[V_i]$ is a $\gamma p$-expander,
    \item\label{s1} $\delta(G[V_i]) \geq np/k^2 \geq 2\alpha np$,
    \item $\delta_{\xi}(G[V_i]) \geq (1/k + \alpha - \xi)np$.
  \end{enumerate}

  Let $n_i := |V_i|$ and $k_i := \ceil{k \cdot n_i/n}$, for every $i \in
  [\ell]$. Then, by our choice of constants,
  \[
    \delta_{\alpha'/128}(G[V_i]) = \delta_{\xi}(G[V_i]) \geq (1/k + \alpha')np
    \geq (1/k_i + \alpha')n_ip.
  \]
  Next, $\ref{s1}$ and the fact that $G$ is
  $(p, \beta)$-sparse imply that
  \[
    \alpha |V_i|np \leq e(V_i, V_i) \leq |V_i|^2p + \beta|V_i| = |V_i|(|V_i|p +
    \beta),
  \]
  from which we deduce, using $\beta \leq \eta np/\log^3 n$ for large enough
  $n$, that $n_i \geq (\alpha/2)n$. In particular, we have $\beta \leq \eta
  np/\log^3 n \leq \eta'n_ip/\log^3 n_i$, by our choice of $\eta$. Since
  $G[V_i]$ is a subgraph of $G$, it is clearly also $(p, \beta)$-sparse.

  It now follows from Theorem~\ref{thm:main-expander} applied with $\alpha'$ (as
  $\alpha$) that each graph $G[V_i]$ can be covered by $k_i - 1$ cycles. Since
  \[
    \sum_{i = 1}^{\ell} (k_i - 1) = \sum_{i = 1}^{\ell} \ceil{k \cdot n_i/n} - 1
    < \sum_{i = 1}^{\ell} k \cdot n_i/n = k,
  \]
  this shows that also $G$ can be covered by $k - 1$ cycles, as required.
\end{proof}

It thus remains to prove Theorem~\ref{thm:main-expander}. As one might expect,
the proof makes heavy use of the fact that the graph $G$ is a $\gamma
p$-expander. The main ingredient is the Connecting Lemma which allows us to
connect many given pairs of vertices in $G$ using short disjoint paths. In order
to make this precise, we start with the following definition.

\begin{definition}[$(M, W, \ell)$-matching]
  Let $G$ be a graph and let $W \subseteq V(G)$ be a subset of the vertices. Let
  $M$ be a multigraph with $V(M) \subseteq V(G) \setminus W$. Then an \emph{$(M,
  W, \ell)$-matching} in $G$ is a collection $\{P_e : e \in E(M)\}$ of
  internally vertex-disjoint paths in $G$ where for every edge $e = \{u, v\} \in
  E(M)$, the path $P_e$ is a $uv$-path of length at most $\ell$ whose internal
  vertices all lie in the set $W$.
\end{definition}

The multigraph $M$ can be thought of as prescribing which pairs of vertices of
$G$ ought to be connected by how many paths; an $(M, W, \ell)$-matching is then
a system of internally vertex-disjoint paths of length at most $\ell$ connecting
the prescribed vertices in the prescribed manner, such that the internal
vertices of the paths all lie in the set $W$. The proof of the Connecting Lemma
is deferred to Section~\ref{sec:connecting-lemma}.

\begin{restatable}[Connecting Lemma]{lemma}{connecting}
  \label{lem:connecting-lemma}
  For every $\gamma \in (0, 1)$ and $\Delta > 0$, there exists a positive
  $C(\gamma, \Delta)$ such that the following holds for all sufficiently large
  $n$. Let $p \in (0, 1)$ and $\beta > 0$, let $G$ be a $(p, \beta)$-sparse
  graph on $n$ vertices, and let $U, W \subseteq V(G)$ be disjoint subsets such
  that:
  \begin{enumerate}[(i)]
    \item\label{cl-W-exp} $G[W]$ is a $\gamma p$-expander,
    \item\label{cl-W-size} $|W| \geq C\beta\sqrt{\log n}/p$, and
    \item\label{cl-U-deg} every vertex $u \in U$ satisfies $\deg(u, W) \geq
      \gamma |W|p$.
  \end{enumerate}
  Let $\ell = \ceil{30\log n/(\gamma\log\log n)}$. Then for every multigraph $M$
  with $V(M) \subseteq U$, at most $|W|/(C\ell)$ edges, and maximum degree at
  most $\Delta$, there exists an $(M, W, \ell)$-matching in $G$.
\end{restatable}

One important technical property that we make use of is that most not too small
subsets of a good expander again induce good expanders, albeit with a slightly
weaker parameter.

\begin{restatable}[Inheritance Lemma]{lemma}{inheritance}
  \label{lem:inheritance-lemma}
  For every $\gamma \in (0, 1)$, there exist positive $\gamma_1(\gamma)$ and
  $C(\gamma)$ such that the following holds for all sufficiently large $n$. Let
  $p \in (0, 1)$ and $\beta > 0$, let $G$ be a $(p, \beta)$-sparse $\gamma
  p$-expander on $n$ vertices, and let $r$ be a positive integer such that $rp
  \geq C \cdot \max{\{ \log n, \beta \}}$. Then the number of subsets $R
  \subseteq V(G)$ of size $r$ for which $G[R]$ is a $\gamma_1p$-expander is at
  least $(1 - n^{-1}) \binom{n}{r}$.
\end{restatable}

The proof of the lemma is based on the sparse regularity lemma. We postpone it
to Section~\ref{sec:inheritance}. Here we just remark that the proof gives a
dependence in the order of $\gamma_1 = \Theta(\gamma^4)$, but we have no reason
to believe that this is optimal. An interesting question, unrelated to the topic
of this paper, is whether it is possible to achieve a dependence of the form
$\gamma_1 = \Omega(\gamma)$.

The proof of Theorem~\ref{thm:main-expander} relies on the {\em absorbing
method}. The main idea of this approach is to embed a small auxiliary structure
(an \emph{absorber}) into the graph $G$ that allows us to reduce the problem of
covering $G$ by $k - 1$ cycles to a simpler problem. In our case, this simpler
problem is to show that there exist subgraphs $P_1, \dotsc, P_{k-1} \subseteq G$
covering the vertices of $G$, where each $P_i$ is a vertex-disjoint union of
$O(n/\log^3 n)$ paths. This simpler problem can then be solved using a method
based on the sparse regularity lemma.

\begin{restatable}[$(X, Y)$-absorber]{definition}{absorber}
  \label{def:absorber}
  Let $X$ and $Y$ be disjoint sets of vertices. An \emph{$(X, Y)$-absorber} is a
  graph $H$ with two designated vertices $a$ and $b$ (called the
  \emph{endpoints} of the absorber) such that $X \cup Y \subseteq V(H) \setminus
  \{a, b\}$ and such that, for all subsets $X' \subseteq X$ and $Y' \subseteq Y$
  with $|X'| = |Y'|$, $H$ contains an $ab$-path $P$ with $V(P) = V(H) \setminus
  (X' \cup Y')$ (i.e., an $ab$-path using all vertices of $H$ with the exception
  of exactly the vertices in $X' \cup Y'$).
\end{restatable}

With this definition we can state the Absorbing Lemma whose proof is presented
in Section~\ref{sec:absorbers}.

\begin{restatable}[Absorbing Lemma]{lemma}{absorbing}\label{lem:absorbing-lemma}
  For every $\gamma\in (0, 1)$, there exists a positive $C(\gamma)$ such that
  the following holds for all sufficiently large $n$. Let $p \in (0, 1)$ and
  $\beta > 0$, let $G = (A, B, E)$ be a bipartite $(p, \beta)$-sparse graph on
  $n$ vertices, and let $U, W \subseteq V(G)$ be disjoint subsets such that:
  \begin{enumerate}[(i)]
    \item\label{abs-W-exp} $G[W]$ is a $\gamma p$-expander,
    \item\label{abs-W-size} $|W| \geq C \cdot \max{\{\log n/p, \beta \sqrt{\log
      n}/p, |U|\log^2n\}}$, and
    \item\label{abs-U-deg} $|U| \geq 2$ and every vertex $u \in U$ satisfies
      $\deg (u, W) \geq \gamma |W|p$.
  \end{enumerate}
  Then $G[U \cup W]$ contains a $(U \cap A, U \cap B)$-absorber with one
  endpoint in $W \cap A$ and another in $W \cap B$.
\end{restatable}

Often when dealing with embedding problems in sparse (pseudorandom) graphs it is
not too difficult to embed a desired structure that covers all but $\eps n$
vertices of the graph. With the next lemma we reduce this leftover
significantly, and what is more, require that only the majority of vertices have
the `correct' degree for the structure to be embedded. The proof relies on a
standard application of the sparse regularity lemma combined with a trick of
Nenadov and the second author and is presented in
Section~\ref{sec:embedding-lemma}.

\begin{restatable}[Embedding Lemma]{lemma}{embedding}\label{lem:embedding-lemma}
  For every integer $k \geq 2$ and every $\alpha > 0$, there exists a positive
  $\eta(\alpha, k)$ such that the following holds for all sufficiently large
  $n$. Let $p \in (0, 1)$ and $\beta \leq \eta np$, and let $G$ be a $(p,
  \beta)$-sparse graph on $n$ vertices such that
  \[
    \delta(G) \geq 2\alpha np \quad \text{and} \quad \delta_{\alpha/32}(G) \geq
    (1/k + \alpha) np.
  \]
  Then the vertices of $G$ can be covered by $k - 1$ path forests $P_1, \dotsc,
  P_{k-1}$ that contain at most $\max{\{ \beta/p, \log n/p \}}$ paths each.
\end{restatable}

With all the `big guns' at hand we are ready to prove
Theorem~\ref{thm:main-expander}.

\subsection{Overview}

The proof follows a standard strategy relying on the absorbing method: (1)
choose a partition $V(G) = V' \cup U \cup W$ into sets of appropriate size
uniformly at random; (2) use the Absorbing Lemma to find an absorber $H$ in $G[U
\cup W]$; (3) use the Embedding Lemma to cover the vertices of $G[V' \cup W]$
not in $H$ by $k-1$ path forests each consisting of at most $n/\log^3 n$ paths;
(4) for every forest, use the Connecting Lemma over $G[U]$ to combine all its
paths into a cycle; and (5) use the absorbing property of $H$ to take in the
unused vertices of $G[U]$. One of the main issues we need to deal with while
executing this strategy is that for $k \geq 3$ the graph $G$ can be bipartite.
This introduces several complications.

First of all, in order to find an absorber $H$, the underlying graph $G[U \cup
W]$ needs to be {\em bipartite} and $G[W]$ a {\em $\gamma p$-expander}; even
more, for the Connecting Lemma we need $G[U]$ to be an expander as well. In case
$k = 2$ this is circumvented by choosing a random equipartition of both $U = U_A
\cup U_B$ and $W = W_A \cup W_B$ taking $F$ to be the bipartite graph between
corresponding colour classes. As the (essential) minimum degree of $F[U]$ is in
this case around $(1/2+o(1))|U|p/2$, a simple calculation shows that $F[U]$ is a
$\gamma_1p$-expander, for a suitable choice of $\gamma_1$. The same argument
applies for $F[W]$. The second issue lies in the fact that, when using the
Connecting Lemma across $F[U]$, we need to use exactly the same number of
vertices of both $U_A$ and $U_B$ in order for the `absorbing property' of $H$ to
apply. In case $k = 2$ this is not a problem, as all vertices $v \in V(G)$ have
a significant portion of their degree into {\em both} $U_A$ and $U_B$ in $G$.
One can then ensure that the vertices of $U_A$ and $U_B$ are used in a balanced
manner.

More serious problems start to occur when $k \geq 3$. For a start, the above
strategy of choosing a random equipartition of $U$ and $W$ simply does not work
as the obtained bipartite graph is not necessarily an expander. Before
partitioning $V(G)$ we first apply
Lemma~\ref{lem:expander-to-bipartite-expander} to the whole graph $G$, to find a
spanning bipartite $(\gamma p/2)$-expander $F \subseteq G$ on colour classes $A$
and $B$. Next, a random partition is chosen into sets $V' \cup U \cup W$ as
above, and the Inheritance Lemma (Lemma~\ref{lem:inheritance-lemma}) implies
$F[U]$ and $F[W]$ are w.h.p.\ both a $\gamma_1 p$-expander. However, in this
case we cannot guarantee that there is even a single vertex with significant
portion of its degree into both $U \cap A$ and $U \cap B$---recall, $G$ may be
bipartite itself and $F = G$ potentially! Hence, if in step (4) from above the
Connecting Lemma is applied over $F[U \cap A, U \cap B]$ blindly, it may result
in an imbalance in the number of vertices used, as we cannot control which of
the vertices (the ones given as $V(M)$ to Lemma~\ref{lem:connecting-lemma}) have
neighbours in $U \cap A$ and which in $U \cap B$. It may be that all
$(k-1)n/\log^3 n$ pairs of vertices we are trying to connect have neighbours
only in, say, $U \cap A$, resulting in $(k-1)n/\log^3 n$ more vertices of $U
\cap A$ being used than those of $U \cap B$, since all paths in $F[U]$ with both
endpoints in $U \cap A$ are of even length. Luckily, as $k \geq 3$, we can reuse
some vertices across the $k - 1$ cycles we are trying to find, which gives us
additional flexibility.

The plan is to find sets $Q_A$ and $Q_B$, in $V(G) \setminus (U \cup W)$, each
of size roughly $kn/\log^3 n$, such that all vertices of $Q_A$ in $G$ have a
large neighbourhood in $U \cap A$ and all vertices of $Q_B$ have a large
neighbourhood in $U \cap B$. The Connecting Lemma is then used to connect all
these vertices into a single path $P_Q$, before covering the remainder $V'
\setminus V(P_Q)$ by $k-1$ path forests, each containing at most $n/\log^3 n$
paths. We connect $P_Q$ with the first path forest through $F[U]$ to get a cycle
$C_1$ and then, while connecting the remaining forests into cycles $C_2, \dotsc,
C_{k-1}$, we can \emph{reuse} the vertices of $P_Q$ by adding a carefully chosen
subset of them to the pairs in $V(M)$ we aim to connect by the Connecting Lemma.
This can be done so that exactly the same number of vertices of $U \cap A$ and
$U \cap B$ is used, allowing us to use the `absorbing property' of $H$ to
finally complete the embedding.

\begin{proof}[Proof of Theorem~\ref{thm:main-expander}]
  Without loss of generality we may assume $0 < p \leq 1/2$. Given $k$, $\alpha$,
  and $\gamma$, let $\gamma_1 = \min{\{\alpha/1024, \gamma/1024,
  \gamma_{1_{\ref{lem:inheritance-lemma}}}(\gamma/2)\}}$, $\eps =
  \min{\{\alpha/1024, \gamma_1/64\}}$,
  \[
    C = \max{\{
      30/\eps^2,
      C_{\ref{lem:connecting-lemma}}(\gamma_1, 2),
      C_{\ref{lem:connecting-lemma}}(\gamma_1^2/8, 2),
      C_{\ref{lem:inheritance-lemma}}(\gamma/2),
      C_{\ref{lem:absorbing-lemma}}(\gamma_1/2)
    \}},
  \]
  and $\eta = \eps/C^2$. Suppose that $G = (V, E)$ is a $(p, \beta)$-sparse
  graph with minimum degree at least $2\alpha np$ and $\delta_{\alpha/128}(G)
  \geq (1/k + \alpha)np$. Note that, since $\beta \leq \eta np/\log^3 n$ and
  $\delta(G) \geq 2\alpha np$, the definition of $(p,\beta)$-sparse graphs
  applied to a vertex $v$ and its neighbourhood $N_G(v)$ implies $p \geq
  \eta^{-1}\log^4 n/n$, with room to spare.

  \subsubsection*{The Hamiltonian case, \texorpdfstring{$k = 2$}{k = 2}}

  Let $V = V' \cup U \cup W$ be a partition of $V$ chosen uniformly at random
  such that
  \begin{equation}\label{eq:main-2-sizes}
    |U| = \floor*{\frac{\eps n}{C\log^2 n}}, \quad |W| = \floor{\eps n},
    \quad \text{and} \quad |V'| = n - |U| - |W|.
  \end{equation}
  Hence, from the bounds on $p$ and $\beta$, we have
  \begin{equation}\label{eq:main-2-partition-lb-size}
    |U|, |W|, |V'| \geq C \cdot \max{\{ \beta\sqrt{\log n}/p, \log n/p \}},
  \end{equation}
  by our choice of $\eta$. As an easy consequence of Chernoff's inequality and
  the union bound, with high probability
  \begin{equation}\label{eq:main-2-degrees}
    \deg_G(v, Z) \geq (2\alpha-\eps)|Z|p \qquad \text{and} \qquad
    \delta_{\alpha/126}(G[Z]) \geq (1/2+3\alpha/4)|Z|p,
  \end{equation}
  for all $v \in V(G)$ and $Z \in \{V', U, W\}$. For the remainder of the
  proof we fix such a good choice of sets $V'$, $U$, and $W$.

  Let $U = U_A \cup U_B$ and $W = W_A \cup W_B$ be partitions with $|U_A| =
  \floor{|U|/2}$ and $|W_A| = \floor{|W|/2}$, and let $F := G[U_A \cup W_A, U_B
  \cup W_B]$ a bipartite graph such that the following holds:
  \begin{enumerate}[label=(P\arabic*), leftmargin=3em]
    \item\label{2-expander} $F[U]$ and $F[W]$ are both a $\gamma_1p$-expander,
    \item\label{2-min-deg} $\deg_F(v, Z) \geq (\alpha/2)|Z|p$ and
      $\delta_{\alpha/120}(F[Z]) \geq (1/4+\alpha/2)|Z|p$, for all $v \in U \cup
      W$ and $Z \in \{U, W\}$,
    \item\label{2-min-deg-fix} for every $v \in V(G)$, $\min\{\deg_G(v, U_A),
      \deg_G(u, U_B)\} \geq \gamma_1|U|p$.
  \end{enumerate}
  If such partitions are chosen uniformly at random then Chernoff's inequality
  and the union bound, together with \eqref{eq:main-2-sizes} and
  \eqref{eq:main-2-degrees}, show that w.h.p.\ \ref{2-min-deg} and
  \ref{2-min-deg-fix} hold. Fix such a choice of sets. We show that these imply
  \ref{2-expander} as well.

  We only show that $F[U]$ is a $\gamma_1p$-expander, as an analogous argument
  works for $F[W]$. Let $X \subseteq U$ be of size $|X| \leq |U|/2$. If $|X|
  \leq (\alpha/4)|U|$, then from \ref{2-min-deg} by applying
  Lemma~\ref{lem:edges-out-small-set} with $\alpha/4$ (as $\alpha$), $F[U]$ (as
  $G$), and $X$ (as $A$) we obtain
  \[
    e_F(X, U \setminus X) \geq \big((\alpha/4)|U|p - \beta\big)|X| \geq
    (\alpha/8)|X||U|p \geq \gamma_1|X||U|p,
  \]
  since $\beta = o(|U|p)$. On the other hand, if $(\alpha/4)|U| < |X| \leq
  |U|/2$, from \ref{2-min-deg} we have
  \[
    e_F(X, U \setminus X) \geq \big(|X| -
    \tfrac{\alpha}{120}|U|\big)(1/4+\alpha/2)|U|p - 2e_G(X \cap U_A, X \cap
    U_B).
  \]
  Then, $G$ being $(p, \beta)$-sparse further gives
  \[
    e_F(X, U \setminus X) \geq (1/4+\alpha/2)|X||U|p -
    \tfrac{\alpha}{120}(1/4+\alpha/2)|U|^2p - |X|^2p/2 - \beta|X|.
  \]
  An easy case distinction between $|X| < |U|/4$ and $|U|/4 \leq |X| \leq |U|/2$
  in both cases shows (with room to spare) $e_F(X, U \setminus X) \geq
  \gamma_1|X||U|p$, as desired.

  Given such a graph $F$, for $X \in \{A, B\}$, we say that a vertex $v \in
  V(G)$ is {\em $X$-expanding} if $\deg_G(v, U_X) \geq \gamma_1|U|p$. We say
  that a pair $\{u, v\}$ of vertices $u, v \in V(G)$ is an {\em $X$-pair} if $u$
  and $v$ are both $X$-expanding. Note that a vertex can be both $A$-expanding
  and $B$-expanding.

  We apply the Absorbing Lemma (Lemma~\ref{lem:absorbing-lemma}) with $\gamma_1$
  (as $\gamma$) and $F$ (as $G$) to obtain a $(U_A, U_B)$-absorber $H$ in the
  graph $F[U \cup W]$ with one endpoint $a \in W_A$ and the other $b \in W_B$.
  Indeed, \ref{2-expander}, \eqref{eq:main-2-sizes} and
  \eqref{eq:main-2-partition-lb-size}, and \ref{2-min-deg} in that order verify
  Lemma~\ref{lem:absorbing-lemma}~$\ref{abs-W-exp}$--$\ref{abs-U-deg}$.

  Let $V'' := V' \cup (W \setminus V(H))$. Observe that $|W \setminus V(H)| \leq
  \eps n$. Hence, we get for all $v \in V''$
  \begin{equation}\label{eq:min-deg-cov-V}
    \deg_G(v, V'') \geq \deg_G(v, V') \osref{\eqref{eq:main-2-degrees}}\geq
    (2\alpha-\eps)|V'|p \geq \frac{2\alpha-\eps}{1+2\eps}|V''|p \geq
    \alpha|V''|p.
  \end{equation}
  Similarly, all $v \in V''$ that previously satisfied $\deg_G(v, V') \geq
  (1/2+3\alpha/4)|V'|p$, now satisfy
  \begin{equation}\label{eq:essential-min-deg-cov-V}
    \deg_G(v, V'') \geq \frac{1/2+3\alpha/4}{1+2\eps}|V''|p \geq
    (1/2+\alpha/2)|V''|p,
  \end{equation}
  due to our choice of $\eps$. Lastly, the set of vertices in $V''$ violating
  \eqref{eq:essential-min-deg-cov-V} is of size at most $(\alpha/126)|V'| +
  |W| \leq (\alpha/64)|V''|$. Therefore, we can apply the Embedding Lemma
  (Lemma~\ref{lem:embedding-lemma}) with $\alpha/2$ (as $\alpha$) to the graph
  $G [V'']$ to obtain a path forest $P_1$ that contains at most $\max{\{
  \beta/p, \log n/p \}} \leq n/\log^3 n$ paths and covers all vertices of $V''$.

  Let $m$ denote the total number of paths in $P_1$. Take an arbitrary ordering
  of these $m$ paths and let us denote the endpoints of the $i$-th path by $s_i$
  and $t_i$, for all $i \in [m]$; note that a path may be only a single vertex,
  in which case $s_i = t_i$. Since every $v \in V''$ has degree at least
  $\gamma_1|U|p$ into {\em both} $U_A$ and $U_B$ (by \ref{2-min-deg-fix}), by
  sequentially choosing for each vertex whether we make it $A$-expanding or
  $B$-expanding, that is, artificially removing the edges to the other colour
  class of $F[U]$, we can make the following set of pairs have exactly the same
  number of $A$-pairs as $B$-pairs:
  \[
    \cP = \big\{\{b, s_1\}, \{t_i, s_{i + 1}\}_{i \in [m - 1]}, \{t_m,
    a\}\big\}.
  \]

  Finally, we apply the Connecting Lemma (Lemma~\ref{lem:connecting-lemma}) to
  the graph obtained from $G$ by substituting $G[U]$ by $F[U]$, and with
  $\gamma_1$ (as $\gamma$), $V(G) \setminus U$ (as $U$), $U$ (as $W$), $\ell =
  30\log n/(\gamma_1 \log\log n)$, and the multigraph $M$ with the vertex set
  \[
    V(M) = \{a, b\} \cup \{s_i, t_i\}_{i \in [m]}
  \]
  and the edge set $E(M) = \cP$. Clearly, $\Delta(M) \leq 2$ and $e(M) \leq m+1
  \leq 2n/\log^3 n \leq |U|/(C\ell)$. Therefore, we obtain an $ab$-path $P'$
  covering all vertices of $V''$. Crucially, $|V(P') \cap U_A| = |V(P') \cap
  U_B|$, ensured by the argument above.

  Let $A'$ and $B'$ denote the set of vertices in $U_A$ and $U_B$, respectively,
  belonging to $V(P')$. By definition (see Definition~\ref{def:absorber}), $H$
  contains an $ab$-path $P$ with $V(P) = V(H) \setminus (A' \cup B')$, and
  substituting $H$ with $P$ together with $P'$ closes a cycle $C_1$ which covers
  all vertices of $G$, as required.

  \subsubsection*{The general case, \texorpdfstring{$k \geq 3$}{k = 3}}

  In this case, due to low (essential) minimum degree, splitting $W$ and $U$
  uniformly at random into colour classes as above is not enough to guarantee
  the obtained bipartite graph is an expander. On top of that, the whole graph
  $G$ may be bipartite and a property like \ref{2-min-deg-fix} cannot hold. We
  use a slightly different approach in this case.

  Let $F \subseteq G$ be a spanning bipartite $(\gamma p/2)$-expander with
  colour classes $A$ and $B$ obtained by applying
  Lemma~\ref{lem:expander-to-bipartite-expander} to $G$. Let $V = V' \cup U \cup
  W \cup Q \cup Y$ be a partition of $V$ chosen uniformly at random such that
  \begin{equation}\label{eq:main-k-set-sizes}
    |U| = \floor*{\frac{\eps n}{C\log^2 n}}, \quad |W| = |Q| = |Y| = \floor{\eps
    n}, \quad \text{and} \quad |V'| = n - |U| - |W| - |Q| - |Y|.
  \end{equation}
  Observe that, from the bounds on $p$ and $\beta$, we have
  \begin{equation}\label{eq:main-k-partition-lb-size}
    |U|, |W|, |Q|, |Y|, |V'| \geq C \cdot \max{\{\beta\sqrt{\log n}/p, \log
    n/p\}},
  \end{equation}
  as before. Since $\delta(F) \geq (\gamma/2)p(n-1)$, as an easy consequence of
  Chernoff's inequality and the union bound, with high probability
  \begin{equation}\label{eq:main-k-degrees}
    \deg_F(v, Z) \geq \gamma_1|Z|p, \enspace \deg_G(v, Z) \geq
    (2\alpha-\eps)|Z|p, \enspace \text{and} \enspace \delta_{\alpha/126}(G[V'])
    \geq (1/k+3\alpha/4)|V'|p,
  \end{equation}
  for every $v \in V(G)$ and $Z \in \{V', U, W, Q, Y\}$. Additionally, from the
  Inheritance Lemma (Lemma~\ref{lem:inheritance-lemma}) with probability at
  least $1 - 6n^{-1}$ we have that all of $G[Q]$, $G[Y]$, $G[U \cup Q]$, $F[U]$,
  $F[W]$, $F[U \cup W]$ are a $\gamma_1p$-expander. For the remainder of the
  proof we fix such a good partition.

  By the Absorbing Lemma (Lemma~\ref{lem:absorbing-lemma}) applied with
  $\gamma_1$ (as $\gamma$) and $F$ (as $G$) there is a $(U \cap A, U \cap
  B)$-absorber $H$ in the graph $F[U \cup W]$ with endpoints $a \in W \cap A$
  and $b \in W \cap B$. Indeed, $F[W]$ being a $\gamma_1p$-expander,
  \eqref{eq:main-k-set-sizes} and \eqref{eq:main-k-partition-lb-size}, and
  \eqref{eq:main-k-degrees} in that order establish
  Lemma~\ref{lem:absorbing-lemma}~$\ref{abs-W-exp}$--$\ref{abs-U-deg}$.

  As before, given $U \cap A$ and $U \cap B$ we say that a pair $\{u, v\}$ of
  vertices $u, v \in V(G)$ is an $A$-pair if $u$ and $v$ are both $A$-expanding
  and a $B$-pair if both are $B$-expanding. In order to gain some flexibility on
  the number of $A$-pairs and $B$-pairs we make use of the following claim.

  \begin{claim}\label{cl:main-k-fix-balance}
    There exist $Q_A, Q_B \subseteq Q$ of size $|Q_A|, |Q_B| = \floor{kn/\log^3
    n}$ such that
    \begin{enumerate}[(a)]
      \item $\deg_G(v, U \cap A) \geq (\gamma_1^2/8)|U|p$ for all $v \in Q_A$,
        and
      \item $\deg_G(v, U \cap B) \geq (\gamma_1^2/8)|U|p$ for all $v \in Q_B$.
    \end{enumerate}
  \end{claim}
  \begin{proof}
    We show only the first assertion as the second one follows analogously. Let
    $U_A := U \cap A$ and $U_B := U \cap B$. We first show that both of these
    sets have size at least $\gamma_1|U|/4$. Indeed, if we assume that, say,
    $|U_A| < \gamma_1|U|/4$, we have from the fact that $F[U]$ is a
    $\gamma_1p$-expander, and thus $\delta(F[U]) \geq \gamma_1(|U|-1)p$,
    \[
      \gamma_1|U|^2p/3 \leq e_F(U_A, U_B) \leq e_G(U_A, U_B) \leq
      \gamma_1|U|^2p/4 + \beta|U|/2,
    \]
    where the upper bound follows from $G$ being $(p,\beta)$-sparse. This leads
    to a contradiction as $\beta \leq \eta np/\log^3 n = o(|U|p)$.

    Let $Q_A \subseteq Q$ be defined as
    \[
      Q_A := \{v \in Q : \deg_G(v, U_A) \geq (\gamma_1^2/8)|U|p\},
    \]
    and suppose towards contradiction $|Q_A| \leq n/\log^2 n$. Using the fact
    that $G[U \cup Q]$ is a $\gamma_1p$-expander and $|U_A| \geq
    (\gamma_1/4)|U|$, we get
    \begin{equation}\label{eq:balance-set-k-3}
      (\gamma_1^2/4)|U||Q \cup U_B|p \leq e_G(U_A, Q \cup U_B) \leq e_G(U_A,
      U_B) + e_G(U_A, Q_A) + (\gamma_1^2/8)|Q||U|p.
    \end{equation}
    As $\beta \leq \eta np/\log^3 n$, $|U_A|, |U_B| = \Theta(n/\log^2 n)$, and
    $|Q| = \eps n$, from $G$ being $(p, \beta)$-sparse, we have
    \[
      e_G(U_A, U_B) \leq (1 + o(1))|U_A||U_B|p \qquad \text{and} \qquad e_G(U_A,
      Q_A) \leq o(|U||Q|p).
    \]
    Lastly, as $|U_B| = o(|Q|)$, the whole right hand side in
    \eqref{eq:balance-set-k-3} can be bounded by $(3\gamma_1^2/16)|Q||U|p$,
    which leads to a contradiction. In conclusion, $|Q_A| > n/\log^2 n$. This
    implies there exist sets $Q_A, Q_B$ of size $\floor{kn/\log^3n}$ as desired.
  \end{proof}

  Take an arbitrary ordering $\{v_1, \dotsc, v_q\}$ of the vertices in $Q_A \cup
  Q_B$, where $q := |Q_A \cup Q_B|$. We apply the Connecting Lemma
  (Lemma~\ref{lem:connecting-lemma}) to $G$ with $\gamma_1$ (as $\gamma$), $Q_A
  \cup Q_B$ (as $U$), $Y$ (as $W$), and the multigraph $M$ defined as
  \[
    V(M) = Q_A \cup Q_B \qquad \text{and} \qquad E(M) = \big\{ \{v_i, v_{i +
    1}\}_{i \in [q - 1]} \big\}.
  \]
  Since $G[Y]$ is a $\gamma_1p$-expander, by \eqref{eq:main-k-partition-lb-size}
  and \eqref{eq:main-k-degrees} all the assumptions
  Lemma~\ref{lem:connecting-lemma}~$\ref{cl-W-exp}$--$\ref{cl-U-deg}$ are
  satisfied. Perhaps the least obvious is the bound on $e(M)$ which holds as $q
  \leq 2k n/\log^3 n \leq |Y|/(C\log n)$
  (see~\eqref{eq:main-k-partition-lb-size}). We obtain a $v_1v_q$-path $P_Q$
  which contains all vertices of $Q_A \cup Q_B$.

  Let $W' := W \setminus V(H)$, $Y' := Y \setminus V(P_Q)$, $Q' := Q \setminus
  (Q_A \cup Q_B)$, and $V'' := V' \cup W' \cup Q' \cup Y'$. One easily checks,
  similarly as in \eqref{eq:min-deg-cov-V} and
  \eqref{eq:essential-min-deg-cov-V}, that $\delta(G[V'']) \geq \alpha|V''|p$
  and $\delta_{\alpha/64}(G[V'']) \geq (1/k+\alpha/2)|V''|p$. Consequently, we
  apply the Embedding Lemma (Lemma~\ref{lem:embedding-lemma}) with $\alpha/2$
  (as $\alpha$) to the graph $G[V'']$ to obtain a collection of $k - 1$ path
  forests $P_1, \dotsc, P_{k - 1}$ which contain at most $n/\log^3 n$ paths
  each.

  Our goal at this point is to connect all the paths belonging to $P_1$ together
  with $P_Q$ into the first cycle $C_1$, connect all the paths belonging to
  $P_2$ together with the absorber $H$ into the second cycle $C_2$, and the
  paths belonging to each $P_i$ for all $i \geq 3$ into a cycle $C_i$. Along the
  way of constructing the cycles $C_2, \dotsc, C_{k-1}$, we may {\em reuse} some
  of the vertices of $Q_A \cup Q_B$, which is not a problem.

  Set $\gamma' := \gamma_1^2/8$. Let $m_i$ denote the number of paths in $P_i$,
  for all $i \in [k - 1]$. Take an arbitrary ordering of the paths in each $P_i$
  and denote their endpoints by $s_i^j$ and $t_i^j$, for all $i \in [k - 1]$, $j
  \in [m_i]$. We again artificially make every vertex either $A$-expanding or
  $B$-expanding depending on their degree into $U \cap A$ and $U \cap B$,
  keeping at least $\gamma'|U|p$ edges for each vertex. This is possible due to
  \eqref{eq:main-k-degrees} and our choice of $\gamma_1$. Next, we construct a
  set $\{u_1, \dotsc, u_x\}$, for some $0 \leq x \leq 2kn/\log^3 n$, by
  sequentially choosing vertices from $Q_A$ or $Q_B$ depending on the difference
  in the number of $A$-pairs and $B$-pairs. It is easy to see that one can
  greedily choose such a set in order for

  \begin{minipage}{\textwidth}
    \begin{minipage}[t]{0.6\textwidth}
      \begin{align*}
        \cP = \big\{ & \{v_q, s_1^1\}, \{t_1^j, s_1^{j + 1}\}_{j \in [m_1 - 1]},
        \{t_1^{m_1}, v_1\}, \\
        & \{t_i^j, s_i^{j + 1}\}_{3 \leq i \leq k - 1, j \in [m_i - 1]},
        \{t_i^{m_i}, s_i^1\}_{i \in [k - 1]} \\
        & \{b, s_2^1\}, \{t_2^j, s_2^{j+1}\}_{j \in [m_2-1]}, \{t_2^{m_2},
        u_1\}, \{u_i, u_{i+1}\}_{i \in [x-1]}, \{u_x, a\}
        \big\}
      \end{align*}
    \end{minipage}
    \begin{minipage}[t]{0.3\textwidth}
      \begin{align*}
        & \text{\small\itshape (cycle $C_1$ including $P_Q$)} \\
        & \text{\small\itshape (cycles $C_3, \dotsc, C_{k - 1}$)} \\
        & \text{\small\itshape (cycle $C_2$ including $H$)} \\
      \end{align*}
    \end{minipage}%
  \end{minipage}
  to be such that the number of $A$-pairs and $B$-pairs in $\cP$ is the same.

  Finally, we apply the Connecting Lemma (Lemma~\ref{lem:connecting-lemma}) to
  the graph obtained from $G$ by substituting $G[U]$ by $F[U]$, and with
  $\gamma'$ (as $\gamma$), $V(G) \setminus U$ (as $U$), $U$ (as $W$), and the
  multigraph $M$ with the vertex set
  \[
    V(M) = \{a, b\} \cup \{s_i^j, t_i^j\}_{i \in [k - 1], j \in [m_i]} \cup
    \{v_i\}_{i \in [q]}
  \]
  and the edge set $E(M) = \cP$. This is indeed possible as $\ref{cl-W-exp}$
  $F[U]$ is a $\gamma_1p$-expander and thus a $\gamma'p$-expander as well,
  $\ref{cl-W-size}$ $|U| = \floor{\eps n/(C\log^2 n)}$ and $\beta \leq \eta
  np/\log^3 n$, and $\ref{cl-U-deg}$ all vertices have degree at least
  $\gamma'|U|p$ into $U$ due to Claim~\ref{cl:main-k-fix-balance} and
  \eqref{eq:main-k-degrees}. Therefore, we obtain $k - 1$ cycles $C_1, C_2,
  \dotsc, C_{k-1}$, where $C_2$ contains $P_Q$ as a subgraph and $C_1$ contains
  $H$ and possibly some vertices of $P_Q$ as a subgraph. As in the case $k = 2$,
  replacing $H$ in $C_1$ by an absorbing $ab$-path $P'$ using all the uncovered
  vertices of $U$ completes the proof.
\end{proof}

In the remainder of the paper we supply the missing proofs of the lemmas stated
above, each section fully being dedicated to one of the lemmas.

\section{The Inheritance Lemma}\label{sec:inheritance}

For the convenience of the reader, we repeat the statement of
Lemma~\ref{lem:inheritance-lemma}.

\inheritance*

\begin{proof}
  Assume that $R \subseteq V(G)$ is a set of size $r$ chosen uniformly at
  random. We aim to show that $R$ induces an $\Omega(p)$-expander with
  probability at least $1 - n^{-1}$ which in turn implies the lemma.

  Note that as $G$ is a $\gamma p$-expander, it has minimum degree at least
  $\gamma p (n - 1)$. Then, as $R$ is a uniformly random subset of size $r \geq
  C \log n/p$ standard application of Chernoff's inequality and the union bound
  shows that with probability at least $1 - o(n^{-2})$ (for $C$ large enough),
  we have
  \begin{equation}\label{eq:inh-min-degree}
    \delta(G[R]) \geq \gamma rp/2.
  \end{equation}

  To show that $G[R]$ is a $\gamma_1p$-expander, we need to verify that for all
  partitions $X \cup Y = R$, we have $e(X, Y) \geq \gamma_1 p |X||Y|$. We first
  show that if \eqref{eq:inh-min-degree} holds, then this is indeed the case for
  all partitions where one of the parts, say $X$, is smaller than $\gamma r/4$.
  For this, we use the fact that $G$, and hence also $G[R]$, is $(p,
  \beta)$-sparse. Applying Lemma \ref{lem:edges-out-small-set} to the graph
  $G[R]$ and using $|X| \leq \gamma r/4$ and \eqref{eq:inh-min-degree}, we
  obtain
  \[
    e(X, Y) \geq \gamma|X|rp/4 - \beta|X|.
  \]
  By choosing $C$ to be sufficiently large and from the assumption that $\beta
  \leq rp/C$ we further get
  \[
    e(X, Y) \geq \gamma |X|rp/4 - \beta |X| \geq \gamma|X|rp/5 \geq \gamma
    |X||Y|p/5.
  \]
  In particular, the expansion property holds with $\gamma_1 = \gamma/5$.

  The argument for the remaining case is more complicated and can be summarised
  as follows. We apply the sparse regularity lemma to the graph $G$, thus
  obtaining an $(\eps, \tilde p)$-regular partition $V_0 \cup V_1 \cup \dotsc
  \cup V_t$, where $\eps$ is a small positive constant and $\tilde p =
  e(G)/\binom{n}{2}$ is the density of $G$. The random set $R$ intersects each
  part of the regular partition in roughly the expected number of vertices.
  Moreover, whenever $(V_i, V_j)$ is an $(\eps, \tilde p)$-regular pair with
  density $d_{ij} \tilde p$, then the pair $(V_i \cap R, V_j \cap R)$ is
  $(\eps', d_{ij}\tilde p)$-\emph{lower-regular} with a slightly weaker
  parameter $\eps'$, which follows from Lemma~\ref{lem:small-subsets-regular}.
  Now if we have a partition $X \cup Y = R$, where $|X|, |Y| \geq \gamma r/4$,
  then we `approximate' it by a partition $X' \cup Y' = V(G) \setminus V_0$ by
  letting $X'$ (resp.\ $Y'$) be the union of the classes in the regular
  partition that intersect $X$ (resp.\ $Y$) in some $\Omega(r)$ vertices (note
  that the sets defined in this way may not be disjoint; there is thus a simple
  cleaning-up step to make sure that this is the case). Since $G$ is a $\gamma
  p$-expander and $V_0$ is small, we know that there are $\gamma p|X'||Y'|$
  edges in $G[X', Y']$, which implies that there are many dense regular pairs
  $(V_i, V_j)$ such that $V_i \subseteq X'$ and $V_j \subseteq Y'$. Because for
  such pairs the pair $(V_i \cap R, V_j \cap R)$ is lower-regular, $(V_i \cap X,
  V_j \cap Y)$ contains many edges, that is, there are many edges in $G[X, Y]$.
  We now give the details.

  Set $\delta = 1/e$ and $\tilde p = e(G)/\binom{n}{2}$, and let $d$ and $\eps'
  < d/2$ be sufficiently smaller than $\gamma$ in order to support the arguments
  that follow. Let $\eps_0 = \eps_{0_{\ref{lem:small-subsets-regular}}}(\eps',
  \delta)$ and $D = D_{\ref{lem:small-subsets-regular}}(\eps')$. Lastly, let
  $\eps$ be sufficiently small with respect to all previously chosen constants,
  let $t$ be as given by Lemma~\ref{lem:sparse-regularity-lemma} applied for
  $\eps$, and assume that $C$ is large enough (in particular, $C$ is much larger
  than $t$). For ease of reference we point out that we have the following
  hierarchy of constants
  \[
    0 < \eps \ll \eps_0 \ll \eps' \ll d \ll \gamma < 1 \ll t \ll C.
  \]
  Observe that $G$ being a $\gamma p$-expander which is $(p, \beta)$-sparse
  implies $\gamma p/2 \leq \tilde p \leq 4p$.

  We apply the sparse regularity lemma (Lemma~\ref{lem:sparse-regularity-lemma})
  to $G$ for $\eps$ to obtain an $(\eps, \tilde p)$-regular partition $(V_i)_{i
  = 0}^{t}$ of $V(G)$, where $V_0$ is the exceptional class of size $|V_0| \leq
  \eps n$. Since $|V_1| = \dotsb = |V_t|$, we have $(1 - \eps)n/t \leq |V_i|
  \leq n/t$ for all $i \in [t]$.

  For each $0 \leq i \leq t$, set $U_i := R \cap V_i$. Note that each $U_i$ is a
  random subset of $V_i$ whose size follows a hypergeometric distribution with
  mean $r|V_i|/n \geq C \log n \cdot \min\{\eps, (1 - \eps)/t\}$. By our choice
  of $C$, Chernoff's inequality together with the union bound shows that with
  probability at least $1 - o(n^{-2})$, all sets $U_i$, $i \in [t]$, satisfy
  $|U_i| = (1 \pm \eps) r|V_i|/n$ and similarly $|U_0| \leq (1 + \eps) \eps r$.
  Recalling that $(1 - \eps) n/t \leq |V_i| \leq n/t$, this implies in
  particular $|U_i| = (1 \pm \eps') r/t$ and $|U_0| \leq \eps' r$.

  Consider now a fixed $(\eps, \tilde p)$-regular pair $(V_i, V_j)$ with density
  at least $d \tilde p$. By definition, $(V_i, V_j)$ is then also $(\eps/d,
  d\tilde p)$-lower-regular. We apply Lemma~\ref{lem:small-subsets-regular} with
  $\eps'$, $\delta$, $d \tilde p$ (as $p$), and $|U_i|, |U_j|$ (as $q_1, q_2$)
  to get that with probability at least
  \[
    1 - \delta^{\min{\{q_1, q_2\}}} = 1 - \delta^{(1 - \eps') C\log n/t} \geq 1
    - o(n^{-2}),
  \]
  the pair $(U_i, U_j)$ is $(\eps', d\tilde p)$-lower-regular. Indeed, we can
  apply the lemma since $|U_i|, |U_j| \geq (1 - \eps')r/t \geq D(d \tilde
  p)^{-1}$ due to the assumption of the lemma on $r$ and the fact that $\gamma
  \tilde p/2 \leq p$, and $\eps/d < \eps_0$. Therefore, the union bound over at
  most $t^2$ pairs $(V_i, V_j)$ shows that with probability at least $1 -
  o(n^{-2})$, the following two properties hold:
  \begin{enumerate}[(i), font=\itshape]
    \item\label{ui-size} $|U_0| \leq \eps' r$ and $(1 - \eps')r/t \leq (1 -
      \eps)r |V_i|/n \leq |U_i| \leq (1 + \eps')r/t$, for all $i\in [t]$, and
    \item\label{ui-lower-reg} whenever $(V_i, V_j)$ is an $(\eps, \tilde
      p)$-regular pair with density $d_{ij} \tilde p \geq d \tilde p$, then
      $(U_i, U_j)$ is $(\eps', d_{ij} \tilde p)$-lower-regular.
  \end{enumerate}
  From these two properties we show the expansion of $G[R]$ deterministically
  for all partitions $X \cup Y = R$ where $|X|, |Y| \geq \gamma r/4$. Let us fix
  such a partition and assume without loss of generality that $|X| \leq r/2 \leq
  |Y|$. As outlined above, we approximate the given partition of $R$ by a
  certain partition of the vertices of $V(G) \setminus V_0$, or, to be more
  precise, by a partition of the classes $V_1, \dotsc, V_t$. We now describe
  precisely how to do this.

  Let $\cI_X = \{i \in [t] : |U_i \cap X| \geq \gamma |U_i|/10\}$ and
  $\cI_Y = \{i \in [t] : |U_i \cap Y|\geq \gamma |U_i|/10\}$. Note that
  $\cI_X \cup \cI_Y = [t]$, although $\cI_X$ and $\cI_Y$ are not necessarily
  disjoint. Using $\ref{ui-size}$ we have
  \[
    |\cI_X| \cdot (1 + \eps')r/t \geq \sum_{i \in \cI_X} |U_i| \geq |X| -
    \sum_{i \notin \cI_X} \gamma |U_i|/10 - |U_0| \geq \gamma r/4 - (1 +
    \eps')\gamma r/10 - \eps'r
  \]
  and hence
  \[
    |\cI_X| \geq \frac{\gamma r/4 - (1 + \eps')\gamma r/10 - \eps' r}{(1 +
    \eps')r/t} \geq \gamma t/8,
  \]
  using the fact that $\eps'$ is sufficiently small compared to $\gamma$. In the
  same way, using $|Y| \geq r/2$, we also obtain $|\cI_Y| \geq t/4$. Since
  $\cI_X \cup \cI_Y = [t]$, it follows that there exists a {\em partition}
  $\cJ_X \cup \cJ_Y = [t]$ such that $\cJ_X \subseteq \cI_X$ and $\cJ_Y
  \subseteq \cI_Y$ and such that $|\cJ_X| \geq \gamma t/8$ and $|\cJ_Y| \geq
  t/4$.

  Next, we say that a pair $(V_i, V_j)$ for $1 \leq i < j \leq t$ is {\em good}
  if it is $(\eps, \tilde p)$-regular with density $d_{ij} \tilde p \geq d
  \tilde p$, and if $(i, j) \in \cJ_X \times \cJ_Y$. From $\ref{ui-lower-reg}$
  it follows that if $(V_i, V_j)$ is good, then $(U_i, U_j)$ is $(\eps', d_{ij}
  \tilde p)$-lower-regular. The lower-regularity and $|U_i \cap X| \geq \gamma
  |U_i|/10 \geq \eps'|U_i|$ and $|U_j \cap Y| \geq \gamma |U_j|/10 \geq
  \eps'|U_j|$ in turn imply
  \[
    e(U_i \cap X, U_j \cap Y) \geq (1-\eps')d_{ij}\tilde p|U_i \cap X||U_j \cap
    Y| \geq \frac{\gamma^2}{100} \cdot (1-\eps')d_{ij}\tilde p|U_i||U_j|.
  \]
  Together with $\ref{ui-size}$, we obtain
  \[
    d_{ij}\tilde p|U_i||U_j| \geq (1-\eps')^3 d_{ij}\tilde p |V_i||V_j| r^2/n^2
    \geq e(V_i, V_j) r^2/(2n^2),
  \]
  which then implies $e(U_i \cap X, U_j \cap Y) \geq \gamma^2 e(V_i, V_j)
  r^2/(200n^2)$. In particular,
  \begin{equation}\label{eq:inh-r1r2}
    e(X, Y) \geq \sum_{\substack{1 \leq i < j \leq t \\ (V_i, V_j) \text{
    good}}} e(U_i \cap X, U_j \cap Y) \geq \frac{\gamma^2 r^2}{200n^2}
    \sum_{\substack{1 \leq i < j \leq t \\ (V_i, V_j) \text{ good}}} e(V_i,
    V_j).
  \end{equation}

  To complete the proof, we show that many edges of $G$ go between good pairs
  $(V_i, V_j)$. Set $X' = \bigcup_{i \in \cJ_X} V_i$ and $Y' = \bigcup_{i \in
  \cJ_Y} V_i$. Note that $X' \cup Y' = V(G) \setminus V_0$. Since $|\cJ_X| \geq
  \gamma t/8$ and $|V_i| \geq (1 - \eps)n/t$, we have $|X'| \geq \gamma n/10$,
  for $\eps$ small enough. Similarly, $|\cJ_Y| \geq t/4$ implies $|Y'| \geq
  n/5$. From the assumption that $G$ is a $\gamma p$-expander, we obtain
  \[
    e(X', Y' \cup V_0) \geq \gamma p |X'||Y'| \geq \gamma^2 n^2p/50.
  \]
  On the other hand, we have
  \[
    e(V_0, X') \leq p|V_0||X'| + \beta \sqrt{|V_0||X'|} \leq 2\eps n^2p.
  \]
  where we used $|V_0| \leq \eps n$ and $\beta \leq rp/C \leq np/C \leq \eps np$
  (provided $C$ is large enough). From these it follows that
  \begin{equation}\label{eq:w1w2}
    e(X', Y') \geq e(X', Y' \cup V_0) - 2\eps n^2p \geq \gamma^2 n^2p/100,
  \end{equation}
  using that $\eps$ is sufficiently small compared to $\gamma$.

  Using once more the fact that $G$ is $(p, \beta)$-sparse, we have $e(V_i, V_j)
  \leq |V_i||V_j|p + \beta \sqrt{|V_i||V_j|} \leq 2n^2p/t^2$ for all $i, j \in
  [t]$. By the definition of an $(\eps, \tilde p)$-regular partition, there are
  at most $\eps t^2$ pairs $(V_i, V_j)$ that are not $(\eps, \tilde p)$-regular.
  Hence, the number of edges in $G[X', Y']$ which go between non-regular pairs
  $(V_i, V_j)$ is at most $2\eps n^2p$. Moreover, the number of edges in $G[X',
  Y']$ with one endpoint in each part of a pair $(V_i, V_j)$ with density below
  $d \tilde p$ is clearly at most $dn^2\tilde p \leq 4dn^2p$. Lastly, note that
  by definition of $X'$ and $Y'$, for every edge $e \in G[X', Y']$, there is
  some $(V_i, V_j)$ with $(i, j) \in \cJ_X \times \cJ_Y$ such that $e \in G[V_i,
  V_j]$. Then \eqref{eq:w1w2} and the definition of a good pair give
  \[
    \sum_{\substack{1 \leq i < j \leq t \\ (V_i, V_j) \text{ good}}} e(V_i,
    V_j) \geq e(X', Y') - 2\eps n^2\tilde p - 4dn^2p \geq \gamma^2
    n^2p/200,
  \]
  using again that $\eps$, $\eps'$, and $d$ are small compared to $\gamma$.
  Finally, with \eqref{eq:inh-r1r2}, we get
  \[
    e(X, Y) \geq \frac{\gamma^2r^2}{200n^2} \sum_{\substack{1 \leq i < j \leq t
    \\ (V_i, V_j) \text{ good}}} e(V_i, V_j) \geq \frac{\gamma^4}{200^2}
    \cdot r^2p \geq \gamma_1 |X||Y|p,
  \]
  for $\gamma_1 = \gamma^4/40000$.
\end{proof}

\section{The Connecting Lemma}\label{sec:connecting-lemma}

The main result of this section is essentially that if $G$ is a pseudorandom
graph and $W$ is a sufficiently large subset inducing a good expander, then $G$
is `highly connected via $W$', i.e., it is possible to connect many given pairs
of vertices using short paths whose internal vertices lie entirely in $W$. We
formally restate the lemma for the convenience of the reader.

\connecting*

The main step towards proving the Connecting Lemma is to prove the following
technical auxiliary lemma whose proof we defer to the end of the section.

\begin{lemma}\label{lem:resilient-connecting}
  For every $\gamma \in (0, 1)$, the following holds for all sufficiently large
  $n$, all $p \in (0, 1)$, and all $\beta > 0$. Let $G$ be a $(p, \beta)$-sparse
  graph on $n$ vertices and let $U, W \subseteq V(G)$ be disjoint subsets such
  that:
  \begin{enumerate}[(i)]
    \item\label{cl-res-W-exp} $G[W]$ is a $\gamma p$-expander,
    \item\label{cl-res-W-size} $|W| \geq 10\gamma^{-1} \beta \sqrt{\log n}/p$,
      and
    \item\label{cl-res-U-deg} every vertex $u \in U$ satisfies $\deg(u, W) \geq
      \gamma|W|p$.
  \end{enumerate}
  Let $\ell = \ceil{10\log n/(\gamma\log\log n)}$. Then for every subset $Z
  \subseteq W$ with
  \[
    |Z| \leq \min{\{\gamma|W|/20, |U|\log n\}},
  \]
  there is a vertex $x \in U$ such that $|N^\ell(x, W \setminus Z)| > |W|/2$.
\end{lemma}

\begin{proof}[Proof of Lemma~\ref{lem:connecting-lemma}]
  We define an $\ell$-uniform hypergraph $\cH$ on the vertex set $E(M) \cup W$,
  where for every edge $e \in E(M)$ and every set $Y \subseteq W$ of size $\ell
  - 1$, we add a hyperedge $\{e\} \cup Y$ if and only if $G$ contains a path
  joining the endpoints of $e$ whose internal vertices belong to $Y$. Hence, if
  there is an $E(M)$-saturating matching in $\cH$, then $G$ contains an $(M, W,
  \ell)$-matching. We use Haxell's criterion (Theorem~\ref{thm:haxell-matching})
  to show that $\cH$ contains an $E(M)$-saturating matching.

  For this, let $S \subseteq E(M)$ and $Z \subseteq W$ be subsets with $|Z| \leq
  2\ell|S|$. Since $M$ has maximum degree at most $\Delta$, we can greedily find
  a subset $S' \subseteq S$ of size $|S|/(2\Delta)$ such that the edges in $S'$
  are pairwise disjoint. In other words, there exist disjoint sets $U_1, U_2
  \subseteq U$ of size $|U_1| = |U_2| \geq |S|/(2\Delta)$, and a bijection $\phi
  \colon U_1 \to U_2$, such that for every $u \in U_1$, the pair $\{u,
  \phi(u)\}$ belongs to $S'$. It is enough to show that for some $u \in U_1$,
  $G$ contains a $u\phi(u)$-path of length at most $\ell$ whose internal
  vertices are all in the set $W \setminus Z$. Indeed, this implies that $\cH$
  contains an edge intersecting $S'$ (and hence $S$) and not intersecting $Z$,
  showing that Haxell's criterion is satisfied. In the remainder we show that
  such a vertex $u \in U_1$ exists.

  Let $\ell_1 = \ceil{10 \log n/(\gamma\log\log n)}$ and let $U_1' \subseteq
  U_1$ be the set consisting of all vertices $u \in U_1$ for which
  $|N^{\ell_1}(u, W \setminus Z)| > |W|/2$. We claim that $|U_1'| > |U_1|/2$.
  Towards a contradiction suppose $|U_1'| \leq |U_1|/2$ and let $U_1'' = U_1
  \setminus U_1'$. Note, $|U_1''| \geq |U_1|/2 \geq |S|/(4\Delta)$. Since $|Z|
  \leq 2\ell|S| \leq |U_1''|\log n$ and also $|Z| \leq 2\ell|S| \leq 2\ell e(M)
  \leq 2|W|/C$, we have
  \[
    |Z| \leq \min{\{ \gamma|W|/20, |U_1''|\log n \}}
  \]
  for sufficiently large $C$. By the assumption $|W| \geq C\beta\sqrt{\log
  n}/p$, we also see that $|W| \geq 10\gamma^{-1} \beta \sqrt{\log n}/p$.
  Therefore, by Lemma~\ref{lem:resilient-connecting} applied for $U_1''$ (as
  $U$), $W$, $Z$, and $\ell_1$ (as $\ell$), there exists a vertex $u \in U_1''$
  such that $|N^{\ell_1}(u, W \setminus Z)| > |W|/2$, which is a contradiction
  with the definition of $U_1''$. It follows that, indeed, $|U_1'| > |U_1|/2$.

  Analogously, the set $U_2' \subseteq U_2$ of all vertices $u \in U_2$ for
  which $|N^{\ell_1}(u, W\setminus Z)| > |W|/2$ has size $|U_2'| > |U_2|/2$. It
  then follows that there is some $u \in U_1'$ such that $\phi(u) \in U_2'$.
  Thus, from both $u$ and $\phi(u)$ it is possible to reach strictly more
  than $|W|/2$ vertices of $W \setminus Z$ via paths of length at most $\ell_1$
  whose internal vertices belong to $W \setminus Z$. Hence, there must exist a
  $u\phi(u)$-path of length at most $2\ell_1 \leq \ell$ whose internal vertices
  all lie in $W \setminus Z$. This completes the proof.
\end{proof}

We now proceed with the proof of Lemma~\ref{lem:resilient-connecting}.

\begin{proof}[Proof of Lemma~\ref{lem:resilient-connecting}]
  Without loss of generality, assume $|U| \leq |W|/(2\log n)$. Indeed, if $U$
  contains more than $|W|/(2\log n)$ elements, then replacing it by a subset of
  size $\floor{|W|/(2\log n)}$ does not violate any of the assumptions
  $\ref{cl-res-W-exp}$--$\ref{cl-res-U-deg}$ nor the bound on the size of $Z$.

  We first prove an auxiliary claim about expansion of certain subsets which we
  use in the proof of the lemma.

  \begin{claim}\label{cl:many-neighbours}
    If $X \subseteq U \cup W$ is such that $|X| \geq |U|$ and $e(X, W \setminus
    X) \geq \gamma|X||W|p/2$, then
    \[
      \big| N(X, W \setminus (X \cup Z)) \big| \geq \min{\{|X|\log n,
      \gamma|W|/5\}}.
    \]
  \end{claim}
  \begin{proof}
    Let us denote $N(X, W \setminus X)$ by $Y$. Recall that since $G$ is $(p,
    \beta)$-sparse, we have
    \[
      e(X, W \setminus X) = e(X, Y) \leq |X||Y|p + \beta\sqrt{|X||Y|}.
    \]
    Combining this with the assumption $e(X, W \setminus X) \geq
    \gamma|X||W|p/2$, we obtain
    \[
      |Y| \geq \frac{\gamma |W|}{2} - \frac{\beta \sqrt{|Y|}}{\sqrt{|X|}p}.
    \]
    If $\sqrt{|Y|} \leq \gamma \sqrt{|X|}|W|p/(4\beta)$, then the above
    inequality gives $|Y| \geq \gamma |W|/4$. In this case, it follows from the
    assumption $|Z| \leq \gamma |W|/20$ that
    \[
      \big| N(X, W \setminus (X \cup Z)) \big| \geq \gamma|W|/4 - |Z| \geq
      \gamma|W|/5.
    \]
    On the other hand, if $\sqrt{|Y|} > \gamma \sqrt{|X|}|W|p/(4\beta)$ then
    using the assumption $|W| \geq 10\gamma^{-1}\beta\sqrt{\log n}/p$ we get the
    lower bound $\sqrt{|Y|} \geq 2\sqrt{|X|\log n}$, or, equivalently, $|Y| \geq
    4|X|\log n$. Since $|Z| \leq |U|\log n \leq |X|\log n$ we obtain
    \[
      \big| N(X, W \setminus (X \cup Z)) \big| \geq 4|X|\log n - |Z| \geq
      |X|\log n,
    \]
    completing the proof of the claim.
  \end{proof}

  We are now ready to prove the lemma. For every integer $j \geq 0$, set
  \[
    \ell_j := \ceil{\log n/(\gamma \log\log n)} + (j + 1) \ceil{5/\gamma}.
  \]
  The goal is to construct a sequence of subsets $U = U_0 \supseteq U_1
  \supseteq U_2 \supseteq \dotsb$ such that for every $j \geq 0$ we have
  $|U_{j+1}| \leq \ceil{|U_j|/\log n}$ and
  \[
    |N^{\ell_j}(U_j, W \setminus Z)| > |W|/2.
  \]
  Note that $|U_j| \leq \ceil{|U_{j - 1}|/\log n}$ implies that either $|U_j| =
  1$ or $|U_j| \leq 2|U_{j - 1}|/\log n$, and therefore, for $j = \ceil{\log
  n/(\log\log n - \log 2)}$, we have $|U_j| = 1$. Moreover, for this $j$, we
  have $\ell_j \leq 10\log n/(\gamma \log\log n) \leq \ell$. In this way, the
  statement of the lemma follows, provided we can indeed construct the sets
  $U_0, U_1, U_2, \dots$ with the properties mentioned above.

  For the set $U_0 = U$, we only need to verify that $|N^{\ell_0}(U_0, W
  \setminus Z)| > |W|/2$. Note that as $e(U_0, W \setminus U_0) \geq
  \gamma|U_0||W|p$ due to $\ref{cl-res-U-deg}$, Claim~\ref{cl:many-neighbours}
  for $U_0$ (as $X$) implies
  \[
    |N^1(U_0, W \setminus Z)| \geq \min{\{ |U|\log n, \gamma|W|/5 \}} \geq |U|,
  \]
  where the last inequality holds because of our assumption $|U| \leq |W|/(2\log
  n)$. Moreover, if for some $i \geq 1$, we have $|U| \leq |N^i(U_0, W \setminus
  Z)| \leq |W|/2$, then as
  \begin{align*}
    e\big( N^i(U_0, W \setminus Z), W \setminus N^i(U_0, W \setminus Z) \big) &
    \osref{$\ref{cl-res-W-exp}$}\geq \gamma \big|N^i(U_0, W \setminus Z)\big|
    \big|W \setminus N^i(U_0, W \setminus Z)\big|p \\
    & \geq \gamma |N^i(U_0, W \setminus Z)||W|p/2,
  \end{align*}
  Claim~\ref{cl:many-neighbours} applied to $N^i(U_0, W \setminus Z)$ (as $X$)
  shows
  \[
    |N^{i+1}(U_0, W \setminus Z)| \geq |N^i(U_0, W \setminus Z)| +
    \min{\{|U|\log^{i+1} n, \gamma|W|/5\}}.
  \]
  One easily sees that for $\ell_0 = \ceil{\log n/(\gamma \log\log n)} +
  \ceil{5/\gamma}$, we have $|N^{\ell_0}(U_0, W \setminus Z)| > |W|/2$ as
  required.

  Suppose now we have constructed the set $U_j$ as above and want to construct
  $U_{j + 1}$. We thus assume
  \[
    |N^{\ell_j}(U_j, W \setminus Z)| > |W|/2.
  \]
  By averaging, there exists a subset $U_{j + 1} \subseteq U_j$ of size $|U_{j +
  1}| \leq \ceil{|U_j|/\log n}$ such that
  \[
    |N^{\ell_j}(U_{j+1}, W \setminus Z)| > |W|/(2\log n) \geq |U|.
  \]
  Successively applying Claim~\ref{cl:many-neighbours} at most $\ceil{5/\gamma}$
  times, it is easy to see that
  \[
    |N^{\ell_j + \ceil{5/\gamma}}(U_{j+1}, W \setminus Z)| > |W|/2.
  \]
  Indeed, in a single step the set $N^{\ell_j}(U_{j+1}, W \setminus Z)$ expands
  to a set of size
  \[
    \min{\{ |W|/(2\log n) \cdot \log n, \gamma |W|/5 \}} \geq \gamma|W|/5,
  \]
  and in at most $\ceil{5/\gamma} - 1$ additional steps, it expands to a set of
  size greater than $|W|/2$. This completes the proof of the lemma.
\end{proof}

\section{The Absorber Lemma}\label{sec:absorbers}

This section is dedicated to the construction of the absorbers and the proof of
the Absorbing Lemma (Lemma~\ref{lem:absorbing-lemma}). We recall the definition
of an absorber first.

\absorber*

The most `natural' way to construct an absorber would be to first find
structures which `absorb' one vertex of $X \cup Y$ at a time, and then depending
on the sets $X'$ and $Y'$ individually decide which vertex needs to be
`absorbed'. Unfortunately, this is not possible in our case for the following
reason. A single-vertex absorber $A_x$ would need to have two $ab$-paths: one
containing $x$, one not containing $x$, and both containing all other vertices
of $A_x$. It is an easy observation that such a structure necessarily contains
an odd cycle. Since the graph $G$ might be bipartite, we cannot hope to find a
single-vertex absorber described above in it.

In order to circumvent this we instead first build a collection of {\em
two-vertex absorbers} $H_{xy}$, each of which contains two paths between its
endpoints, one containing all the vertices of $H_{xy}$ including $x$ and $y$ and
the other containing all vertices of $H_{xy}$ except $x$ and $y$. Such an
absorber is depicted in Figure~\ref{fig:two-vertex-absorber}.

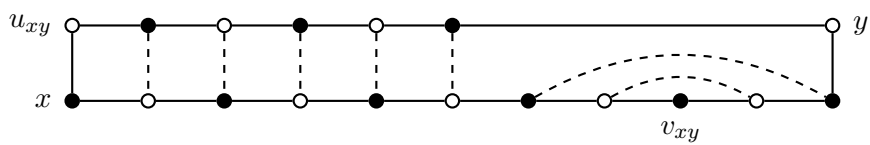
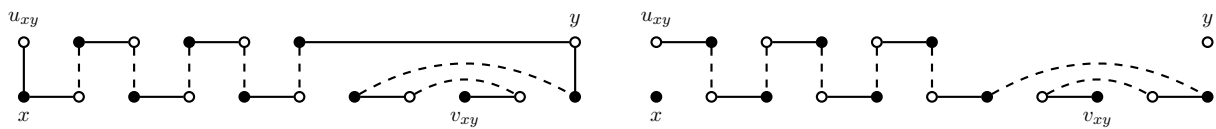
\begin{figure}[!htbp]
  \begin{subfigure}[b]{\textwidth}
    \centering
    \begin{tikzpicture}
  \abspic{3}{5}
  \node[label=left:$x$] at (x){};
  \node[label=right:$y$] at (y){};

  \draw[thick] (x) -- (a1) -- (b1) -- (a2) -- (b2) -- (a3) -- (b3) -- (y);
  \draw[thick] (x) -- (c1) -- (d1) -- (c2) -- (d2) -- (c3) -- (d3) -- (c4) -- (d4) -- (c5) -- (d5) -- (y);

  \node[label=left:$u_{xy}$] at (a1){};
  \node[label=below:$v_{xy}$] at (d4){};
  
  \draw[thick,dashed] (b1) -- (c1) (a2) -- (d1) (b2) -- (c2) (a3) -- (d2) (b3) -- (c3);
  \draw[thick,dashed] (d3) to[bend left] (d5) (c4) to[bend left] (c5);
\end{tikzpicture}
    \caption{An absorber $H_{xy}$ of unrealistically small size.}
  \end{subfigure}
  \par\bigskip
  \begin{subfigure}[b]{0.48\textwidth}
    \centering
    \begin{tikzpicture}[scale=0.725, every node/.style={scale=0.725}]
  \abspic{3}{5}
  \node[label=below:$x$] at (x){};
  \node[label=above:$y$] at (y){};

  \node[label=above:$u_{xy}$] at (a1){};
  \node[label=below:$v_{xy}$] at (d4){};

  \draw[thick,dashed] (b1) -- (c1) (a2) -- (d1) (b2) -- (c2) (a3) -- (d2) (b3) -- (c3);
  \draw[thick,dashed] (d3) to[bend left] (d5) (c4) to[bend left] (c5);

  \draw[thick] (x) -- (a1) (b1) -- (a2) (b2) -- (a3) (b3) -- (y);
  \draw[thick] (x) -- (c1) (d1) -- (c2) (d2) -- (c3) (d3) -- (c4) (d4) -- (c5) (d5) -- (y);
\end{tikzpicture}
    \caption{The `absorbing' $u_{xy}v_{xy}$-path.}
  \end{subfigure}
  \hfill
  \begin{subfigure}[b]{0.48\textwidth}
    \centering
    \begin{tikzpicture}[scale=0.725, every node/.style={scale=0.725}]
  \abspic{3}{5}
  \node[label=below:$x$] at (x){};
  \node[label=above:$y$] at (y){};

  \node[label=above:$u_{xy}$] at (a1){};
  \node[label=below:$v_{xy}$] at (d4){};

  \draw[thick,dashed] (b1) -- (c1) (a2) -- (d1) (b2) -- (c2) (a3) -- (d2) (b3) -- (c3);
  \draw[thick,dashed] (d3) to[bend left] (d5) (c4) to[bend left] (c5);

  \draw[thick] (a1) -- (b1) (a2) -- (b2) (a3) -- (b3) (y);
  \draw[thick] (c1) -- (d1) (c2) -- (d2) (c3) -- (d3) (c4) -- (d4) (c5) -- (d5);
\end{tikzpicture}
    \caption{The `non-absorbing' $u_{xy}v_{xy}$-path.}
  \end{subfigure}
  \caption{An example of a two-vertex absorber for $x$ and $y$. The colours of
  the vertices correspond to the colour classes of the bipartite graph $G$.
  Dashed lines represent disjoint paths of length at most $\ell$ (see the proof
  below) and solid lines the actual edges.}
  \label{fig:two-vertex-absorber}
\end{figure}

The following lemma provides us with a `template' that we use to combine several
two-vertex absorbers into an actual absorber. It is similar to a lemma of
Montgomery~\cite[Lemma~10.7]{montgomery2019spanning}, which is proven in nearly
the same way.

\begin{lemma}\label{lem:template}
  There is an integer $n_0 \in \N$ such that, for every $n \geq n_0$, there exist
  a bipartite graph $G = (A, B, E)$ with $|A| = |B| = 2n$ and $\Delta(G) \leq 40$,
  as well as subsets $A' \subseteq A$ and $B' \subseteq B$ with $|A'| = |B'| = n$,
  satisfying the following. For every set $Z \subseteq A' \cup B'$ with $|Z \cap
  A'| = |Z \cap B'|$, the graph $G[V(G) \setminus Z]$ contains a perfect
  matching.
\end{lemma}
\begin{proof}
  Fix disjoint sets $A_1$ and $B_1$ with $|A_1| = |B_1| = n$. Let $H$ be a
  random bipartite graph with parts $A_1$ and $B_1$ obtained by inserting $20$
  independent random perfect matchings between $A_1$ and $B_1$ (and merging
  eventual multiple edges). We first show that w.h.p.\ $H$ satisfies the
  following properties:
  \begin{enumerate}[(i), font=\itshape]
    \item\label{templatei} for every $X \subseteq A_1$ with $|X| \leq n/2$, we
      have $|N_H(X)| \geq 2|X|$,
    \item\label{templateii} for every $Y \subseteq B_1$ with $|Y| \leq n/2$, we
      have $|N_H(Y)| \geq 2|Y|$, and
    \item\label{templateiii} for any two subsets $X \subseteq A_1$ and $Y
      \subseteq B_1$ with $|X| = |Y| = \ceil{n/4}$, we have $e_H(X, Y) > 0$.
  \end{enumerate}

  Given two sets $X \subseteq A_1$ and $Y \subseteq B_1$, the probability that
  $N_M(X) \subseteq Y$ in a random matching $M$ is $\binom{|Y|}{|X|}
  \binom{n}{|X|}^{-1}$. Therefore, for a fixed integer $t \in [n/4]$, the
  probability that there is a set $X \subseteq A_1$ with $|X| = t$ and $|N_H(X)|
  \leq 2t$ is at most
  \[
    \binom{n}{t} \binom{n}{2t} \bigg( \binom{2t}{t}\binom{n}{t}^{-1} \bigg)^{20}
    \leq \Big( \frac{en}{t} \Big)^t \Big( \frac{en}{2t} \Big)^{2t} \Big(
    \frac{2t}{n} \Big)^{20t} = \bigg(2e^3 \bigg( \frac{2t}{n} \bigg)^{17}
    \bigg)^t.
  \]
  By a union bound over all $t \in [n/4]$, and looking separately at the cases
  $t \leq\log n$ and $\log n \leq t \leq n/4$, the probability that
  $\ref{templatei}$ fails tends to $0$ as $n \to \infty$. In the same way,
  exchanging the roles of $A_1$ and $B_1$, one can show that w.h.p.\ the
  statement in $\ref{templateii}$ holds as well. Lastly, the probability that
  $\ref{templateiii}$ fails is similarly at most
  \[
    \binom{n}{n/4}^2 \bigg( \binom{3n/4}{n/4} \binom{n}{n/4}^{-1} \bigg)^{20}
    \leq (4e)^{n/2} \Big( \frac{3}{4} \Big)^{20n/4} \leq 2^{-n/4}.
  \]
  Thus, w.h.p.\ $\ref{templateiii}$ holds.

  We now take any graph $H$ as above that satisfies $\ref{templatei}$,
  $\ref{templateii}$, and $\ref{templateiii}$---such a graph exists for all
  large enough $n$---and define $G$ by duplicating the vertices in $A_1$ and
  $B_1$ and keeping the edges as in $H$, except that a single edge of $H$ now
  corresponds to four edges in $G$. More precisely, we define $A = (A_1 \times
  \{0\}) \cup (A_1 \times \{1\})$ and $B = (B_1 \times \{0\}) \cup (B_1 \times
  \{1\})$ and insert an edge between $(a, i) \in A$ and $(b, j) \in B$ whenever
  there is an edge between $a$ and $b$ in $H$. We let $A' = A_1 \times \{0\}
  \subseteq A$ and $B' = B_1 \times \{0\} \subseteq B$. Note that because $H$ is
  the union of $20$ matchings, $G$ has maximum degree at most $40$, even after
  duplicating the vertices.

  To complete the proof, we need to show that for any $Z \subseteq A' \cup B'$
  with $|Z \cap A'| = |Z \cap B'|$, the graph $G[V(G) \setminus Z]$ contains a
  perfect matching. Let $m := |Z \cap A'| = |Z \cap B'|$ and set $A_Z := A
  \setminus Z$ and $B_Z := B \setminus Z$. We show that $G[V(G) \setminus Z]$
  contains a perfect matching by verifying Hall's condition, i.e., by showing
  that for every set $X \subseteq A_Z$, we have $|N_G(X, B_Z)| \geq |X|$.

  Assume first $|X| \leq n/2$. Let $X'$ be the larger of the sets $A' \cap X$
  and $(A \setminus A') \cap X$. Then $|X|/2 \leq |X'| \leq n/2$. By
  $\ref{templatei}$ we have $|N_G(X, B_Z)| \geq |N_G(X', B_Z)| \geq |N_H(X')|
  \geq 2|X'| \geq |X|$.

  Suppose next $n/2 < |X| \leq 3n/2 - m$. Let $X'$ be the larger of the sets $A'
  \cap X$ and $(A \setminus A') \cap X$, and note $|X'| \geq n/4$. By
  definition, there are no edges between $X$ and $Y := B_Z \setminus N_G(X,
  B_Z)$. If we assume $|N_G(X, B_Z)| \leq |X|$, then we have $|Y| \geq (2n - m)
  - (3n/2 - m) \geq n/2$. Let $Y'$ be the larger of the sets $B' \cap Y$ and $(B
  \setminus B') \cap Y$. The fact that there are no edges between $X'$ and $Y'$
  then contradicts $\ref{templateiii}$.

  Finally, suppose $3n/2 - m < |X| \leq 2n - m$. Assume towards contradiction
  that $N_G(X, B_Z)$ is contained in a set $Q \subseteq B_Z$ of size $|X| - 1$
  and let $Y := B_Z \setminus Q$. Note that $|Y| = 2n - m - (|X| - 1) \leq n/2$.
  Let $Y'$ be the larger of the sets $B' \cap Y$ and $(B \setminus B') \cap Y$
  and thus $|Y'| \geq |Y|/2$. However, all neighbours of $Y'$ are contained in
  the set $A_Z \setminus X$ of size $2n - m - |X| = |Y| - 1$. This contradicts
  $\ref{templateii}$.
\end{proof}

We are ready to give a proof of the Absorber Lemma. We first restate it for
convenience.

\absorbing*

\begin{proof}[Proof of Lemma~\ref{lem:absorbing-lemma}]
  Let $\gamma_1 = \gamma_{\ref{lem:inheritance-lemma}}(\gamma)$ and $C = 4 \cdot
  \max{\{ C_{\ref{lem:inheritance-lemma}}(\gamma),
  C_{\ref{lem:connecting-lemma}}(\gamma_1, 80) \}}$. Let us write $W_A := W \cap
  A$ and $W_B := W \cap B$. We first show that both of these sets have size at
  least $\gamma |W|/4$. Indeed, if we assume that, say, $|W_A| \leq \gamma
  |W|/4$, we have
  \[
    \gamma|W|^2p/3 \leq e_G(W_A, W_B) \leq p|W_A||W_B| + \beta \sqrt{|W_A||W_B|}
    \leq \gamma p|W|^2/4 + \beta |W| < \gamma p|W|^2/3,
  \]
  where in the first inequality we use the fact that $G$ is a $\gamma
  p$-expander, in the second that it is a $(p, \beta)$-sparse graph, and in the
  last that $|W| \geq C\beta \sqrt{\log n}/p$. Thus, by an analogous argument
  for $|W_B|$, we conclude
  \begin{equation}\label{eq:wawb}
    |W_A| \geq \gamma |W|/4 \qquad \text{and} \qquad |W_B| \geq \gamma |W|/4.
  \end{equation}

  Let now $a$ and $b$ be two arbitrary vertices such that $a \in W_A$ and $b \in
  W_B$. These vertices are going to be the endpoints of our absorber. We start
  by making some preparations.

  Let $m := \max{\{ |U \cap A|, |U \cap B| \}}$ and let $W_A' \subseteq W_A
  \setminus \{a\}$ and $W_B' \subseteq W_B \setminus \{b\}$ be subsets with
  $|W_A'| = 2m - |U \cap A|$ and $|W_B'| = 2m - |U \cap B|$ chosen uniformly at
  random among all subsets of this size. Note that it is possible to choose
  subsets of this size because of \eqref{eq:wawb} and the assumption $|U| \leq
  |W|/(C \log^2 n)$. In the following, we write $U_A := W_A' \cup (U \cap A)$
  and $U_B := W_B' \cup (U \cap B)$; both of these sets have size $2m$.

  Furthermore, let $W_1 \cup W_2 \cup W_3 = W \setminus (W_A' \cup W_B' \cup
  \{a, b\})$ be an equipartition chosen uniformly at random. Since $|W_A'|,
  |W_B'| \leq 2|U| \leq |W|/\log^2 n$, we have $|W_i| \geq |W|/4$. Therefore, by
  our choice of $C$, the assumptions of the lemma together with the Inheritance
  Lemma (Lemma~\ref{lem:inheritance-lemma}), Chernoff's inequality, and the
  union bound, ensure that w.h.p.\ the following three properties are satisfied
  for all $i \in [3]$:
  \begin{enumerate}[(i), font=\itshape]
    \item\label{abs-Wi-exp} $G[W_i]$ is a $\gamma_1p$-expander,
    \item\label{abs-Wi-size} $|W_i| \geq (C/4) \cdot \beta\sqrt{\log n}/p$, and
    \item\label{abs-Wi-deg} for every vertex $u \in U \cup W$, we have $\deg(u,
      W_i) \geq \gamma_1|W_i|p$.
  \end{enumerate}
  In the following we assume that these properties hold deterministically for
  all $W_i$.

  Let $G_T = (A_T, B_T, E_T)$ with $|A_T| = |B_T| = 2m$ and $\Delta(G_T) \leq
  40$ be a graph given by Lemma~\ref{lem:template}. Furthermore, let $A_T'
  \subseteq A_T$ and $B_T' \subseteq B_T$ with $|A_T'| = |B_T'| = m$ be subsets
  given by the former lemma with the property that for every $Z \subseteq A_T'
  \cup B_T'$ with $|Z \cap A_T'| = |Z \cap B_T'|$, the graph $G_T[V(G_T)
  \setminus Z]$ contains a perfect matching. Lastly, let $f \colon V(G_T) \to
  U_A \cup U_B$ be a bijection mapping the vertices of $A_T$ onto $U_A$ and
  $B_T$ onto $U_B$ and such that $U \cap A \subseteq f(A_T')$ and $U \cap B
  \subseteq f(B_T')$.

  The construction of the absorber now proceeds by three independent
  applications of the Connecting Lemma. Set $\ell = \ceil{30\log
  n/(\gamma_1\log\log n)}$. Firstly, we apply it with $\gamma_1$ (as $\gamma$),
  $U_A \cup U_B$ (as $U$), $W_1$ (as $W$), and with the multigraph $M$ with the
  vertex set $U_A \cup U_B$ and the edge set defined as follows: for every edge
  $\{x, y\}\ \in E(G_T)$, add a \emph{double} edge $\{f(x), f(y)\}$ to $M$.
  Since $\Delta(G_T) \leq 40$, it follows that $\Delta(M) \leq 80$. Moreover,
  the assumption $|U|\leq |W|/(C\log^2 n) \leq 4|W_1|/(C\log^2 n)$ shows that
  the assumption $e(M) \leq |W_1|/(C\ell)$ is satisfied. Finally,
  $\ref{abs-W-exp}$--$\ref{abs-U-deg}$ show that assumptions
  Lemma~\ref{lem:connecting-lemma}~$\ref{cl-W-exp}$--$\ref{cl-U-deg}$ are
  satisfied. Therefore, we obtain for every edge $\{x, y\} \in E(G_T)$ two
  $f(x)f(y)$-paths of length at most $\ell$ such that all paths are internally
  vertex-disjoint and use only vertices from $W_1$. For a given edge $\{x, y\}
  \in E(G_T)$, we denote these two $f(x)f(y)$-paths by $P_{xy}$ and $Q_{xy}$ and
  let
  \[
    P_{xy} = f(x), a_p^1, b_p^1, a_p^2, \dotsc, a_p^{\ell_p}, b_p^{\ell_p}, f(y)
    \qquad \text{and} \qquad Q_{xy} = f(x), a_q^1, b_q^1, a_q^2, \dotsc,
    a_q^{\ell_q}, b_q^{\ell_q}, f(y),
  \]
  where $\ell_p = (|P_{xy}| - 2)/2$, $\ell_q = (|Q_{xy}| - 2)/2$, $\ell_q \geq
  \ell_p$ (w.l.o.g.), and $a_p^i, a_q^j \in A$ and $b_p^i, b_q^j \in B$. Note
  that since $f(x)$ and $f(y)$ lie in different colour classes of $G$, both
  paths necessarily have odd length.

  Given a collection of paths $\cP = \{P_{xy}, Q_{xy} : \{x, y\} \in
  E(G_T)\}$ as above, let
  \[
    U_{\cP} := \bigcup_{\{x, y\} \in E_T} V(P_{xy}) \cup V(Q_{xy}) \setminus
    \{f(x), f(y)\}.
  \]
  Next, we apply the Connecting Lemma with $\gamma_1$ (as $\gamma$), $U_\cP$ (as
  $U$), $W_2$ (as $W$), and a multigraph $M$ defined as follows: the vertex set
  of $M$ is just $U_{\cP}$ and the edge set is the union of
  \begin{align*}
    E(M_{xy}) := \big\{ & \{b_p^i, a_q^i\}_{1 \leq i \leq \ell_p}, \{ a_p^{i+1},
      b_q^i \}_{1 \leq i \leq \ell_p - 1}, \\
    & \{b_q^{\ell_p + i - 1}, b_q^{\ell_q - i + 1}\}_{1 \leq i \leq
    \ceil{(\ell_q - \ell_p)/2}} , \{a_q^{\ell_p + i}, a_q^{\ell_q - i + 1}\}_{1
    \leq i \leq \ceil{(\ell_q - \ell_p)/2}} \big\},
  \end{align*}
  for all $\{x, y\} \in E(G_T)$. The edge set is much easier to `define'
  visually---it is given by the dashed lines in
  Figure~\ref{fig:two-vertex-absorber} for every $\{x, y\} \in E(G_T)$.

  It is easy to verify that the assumptions of the Connecting Lemma are all
  satisfied---perhaps the least evident being that $e(M) \leq |W_2|/(C\ell)$,
  which holds because $M$ has at most $e(G_T) \cdot 4\ell \leq 320m \ell \leq
  320|U|\ell$ edges and $|U| \leq |W|/(C\log^2n) \leq |W_2|/(320 C\ell^2)$.
  Therefore, we obtain an $(M, W_2, \ell)$-matching, that is replace the dashed
  edges in Figure~\ref{fig:two-vertex-absorber}, for all $\{x, y\} \in E(G_T)$,
  by internally vertex-disjoint paths in $G$ whose internal vertices all belong
  to $W_2$.

  Lastly, for every $\{x, y\} \in E(G_T)$ denote by $u_{xy}$ and $v_{xy}$ the
  vertices $a_p^1$ and $b_q^{(\ell_p + \ell_q)/2}$ (assuming $\ell_q$ is even,
  otherwise this is $a_q^{(\ell_p + \ell_q + 1)/2})$, respectively (the ones as in
  Figure~\ref{fig:two-vertex-absorber}) and let us fix an arbitrary ordering
  $\{x_1, y_1\}, \dotsc, \{x_{m_T}, y_{m_T}\}$ of the edges of $G_T$, where $m_T
  := e(G_T)$. We now apply the Connecting Lemma for the third and last time with
  $V(M) := \{u_{x_i y_i}, v_{x_i y_i}\}_{i \in [m_T]} \cup \{a, b\}$ (as $U$),
  $W_3$ (as $W$), and the edge set of $M$
  \[
    E(M) := \big\{ \{a, u_{x_ 1y_1}\}, \{v_{x_i y_i}, u_{x_{i + 1} y_{i +
    1}}\}_{i \in [m_T - 1]}, \{v_{x_{m_T} y_{m_T}}, b\} \big\}.
  \]

  Similarly as above, one easily checks that all the assumptions of the
  Connecting Lemma are satisfied and hence we obtain an $(M, W_3,
  \ell)$-matching which connects all the two-vertex absorbers into one large
  absorber $H$. For a more natural, visual representation we depict the obtained
  structure on Figure~\ref{fig:absorber}.

  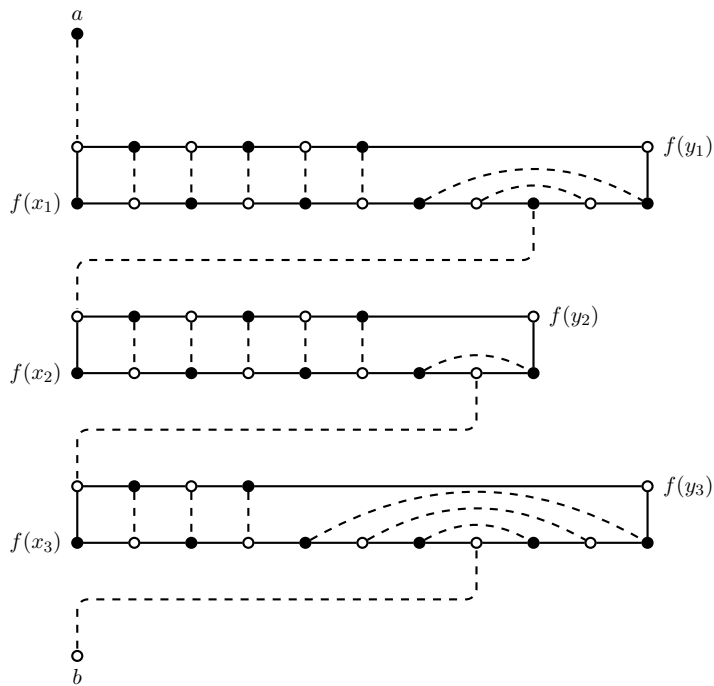
\begin{figure}[!htbp]
    \centering
    \begin{tikzpicture}[scale=0.75, every node/.style={scale=0.75}]
  \abspic{3}{5}
  \draw[thick] (x) -- (a1) -- (b1) -- (a2) -- (b2) -- (a3) -- (b3) -- (y);
  \draw[thick] (x) -- (c1) -- (d1) -- (c2) -- (d2) -- (c3) -- (d3) -- (c4) -- (d4) -- (c5) -- (d5) -- (y);

  \node[label=left:$f(x_1)$] at (x){};
  \node[label=right:$f(y_1)$] at (y){};

  \node (in0) at (a1){};
  \node (out0) at (d4){};

  \draw[thick,dashed] (b1) -- (c1) (a2) -- (d1) (b2) -- (c2) (a3) -- (d2) (b3) -- (c3);
  \draw[thick,dashed] (d3) to[bend left] (d5) (c4) to[bend left] (c5);
  \node[n1][label=above:$a$] (out) at (1,3){};
  \draw[thick,dashed] (out) -- (in0);

  \begin{scope}[yshift=-3cm]
    \abspic{3}{4}
    \node[label=left:$f(x_2)$] at (x){};
    \node[label=right:$f(y_2)$] at (y){};
  
    \node (in1) at (a1){};
    \node (out1) at (c4){};
  
    \draw[thick] (x) -- (a1) -- (b1) -- (a2) -- (b2) -- (a3) -- (b3) -- (y);
    \draw[thick] (x) -- (c1) -- (d1) -- (c2) -- (d2) -- (c3) -- (d3) -- (c4) -- (d4) -- (y);
    \draw[thick,dashed] (b1) -- (c1) (a2) -- (d1) (b2) -- (c2) (a3) -- (d2) (b3) -- (c3);
    \draw[thick,dashed] (d3) to[bend left] (d4);
    \draw[thick,dashed,rounded corners] (out0) -- +(0,-1) -| (in1);
  \end{scope}

  \begin{scope}[yshift=-6cm]
    \abspic{2}{5}
    \node[label=left:$f(x_3)$] at (x){};
    \node[label=right:$f(y_3)$] at (y){};

    \node (in2) at (a1){};
    \node (out2) at (c4){};

    \draw[thick] (x) -- (a1) -- (b1) -- (a2) -- (b2) -- (y);
    \draw[thick] (x) -- (c1) -- (d1) -- (c2) -- (d2) -- (c3) -- (d3) -- (c4) -- (d4) -- (c5) -- (d5) -- (y);
    \draw[thick,dashed] (b1) -- (c1) (a2) -- (d1) (b2) -- (c2);
    \draw[thick,dashed] (d3) to[bend left] (d4) (c3) to[bend left] (c5) (d2) to[bend left] (d5);
    \node[n2][label=below:$b$] (in) at (1,-2){};
    \draw[thick,dashed,rounded corners] (out1) -- +(0,-1) -| (in2);
    \draw[thick,dashed,rounded corners] (out2) -- +(0,-1) -| (in);
  \end{scope}
\end{tikzpicture}
    \caption{The absorber $H$. The vertices $f(x_1), f(x_2), f(x_3)$ and
    respectively $f(y_1), f(y_2), f(y_3)$ are not necessarily distinct (they
    would be distinct only if $G_T$ were itself a perfect matching); however,
    all other vertices are actually distinct.}
    \label{fig:absorber}
  \end{figure}

  It remains to show that the graph constructed in this way is a $(U \cap A, U
  \cap B)$-absorber with endpoints $a$ and $b$. For this, suppose that $A'
  \subseteq U \cap A$ and $B' \subseteq U \cap B$ are subsets such that $|A'| =
  |B'|$. Then we let $Z := f^{-1}(A' \cup B')$ and note that $Z$ is a subset of
  $A_T' \cup B_T'$ that intersects each set $A_T'$ and $B_T'$ in the same number
  of vertices. Hence, by the defining property of $G_T$, the graph $G_T[V(G_T)
  \setminus Z]$ contains a perfect matching $M$. We can then find an $ab$-path
  using all vertices except those in $A' \cup B'$, as follows: for each edge
  $\{x, y\}$ in the given perfect matching take a $u_{xy}v_{xy}$-path which
  includes all vertices of $H_{xy}$; for all other edges take a
  $u_{xy}v_{xy}$-path which includes all vertices of $H_{xy}$ except for $f(x)$
  and $f(y)$. Since the edges in $M$ form a perfect matching, it is clear that
  this is indeed a path (that is, no vertex is used twice) and that this path
  visits each vertex of the absorber except those contained in the set $A' \cup
  B'$.
\end{proof}

\section{The Partitioning Lemma}\label{sec:partitioning-lemma}

In this section we give a proof of the Partitioning Lemma which allows us to
partition every $(p, o(np))$-sparse graph satisfying a certain minimum degree
condition into linear-sized expanders. The proof uses the following notions of
\emph{good} and \emph{perfect} partitions.

\begin{definition}[$(c,\alpha)$-good, $(c,\alpha, \gamma)$-perfect]
  \label{def:good}
  Let $G$ be a graph on $n$ vertices. We say that a partition $V(G) = V_0 \cup
  \dotsb \cup V_\ell$ of the vertex set of $G$ is \emph{$(c,\alpha)$-good} if
  \begin{enumerate}[label=(L\arabic*)]
    \item\label{L-size-bad} $|V_0| \leq \alpha n$, and
    \item\label{L-min-deg} $\delta(G[V_i]) \geq (c + \alpha/2^{\ell})np$ for
      every $i \in [\ell]$.
  \end{enumerate}
  We say that the partition is \emph{$(c, \alpha, \gamma)$-perfect} if it
  additionally satisfies
  \begin{enumerate}[label=(L\arabic*), resume]
    \item\label{L-non-extremal} $G[V_i]$ is a $\gamma p$-expander for every $i
      \in [\ell]$.
  \end{enumerate}
\end{definition}

A first step towards proving Lemma~\ref{lem:partitioning-lemma} is proving the
following auxiliary lemma.

\begin{lemma}\label{lem:perfect-partition}
  For all $c, \alpha \in (0, 1)$, there exist positive $\gamma(\alpha, c)$ and
  $\eta(\alpha, c)$ such that the following holds for all sufficiently large
  $n$. Let $p \in (0, 1)$ and $\beta \leq \eta np$ and let $G$ be a $(p,
  \beta)$-sparse graph with $n$ vertices and minimum degree at least $(c +
  \alpha)np$. Then for some integer $1 \leq \ell < 1/c$, there exists a $(c,
  \alpha, \gamma)$-perfect partition $V(G) = V_0 \cup V_1 \cup \dotsb \cup
  V_\ell$ of the vertex set of $G$.
\end{lemma}
\begin{proof}
  Let $G = (V, E)$ and let $\gamma = \alpha^2/2^{2/c + 1}$. Suppose that $\eta =
  \eta(c, \alpha)$ is sufficiently small for the rest of the argument to go
  through. For every $\ell \geq 1$, set
  \[
    c_\ell = c + \alpha/2^{\ell - 1} \geq c, \quad \alpha_\ell = \alpha (2^{-1}
    + 2^{-2} + \cdots + 2^{-\ell}) \leq \alpha, \quad \text{and} \quad
    \gamma_\ell = \alpha^2/2^{2\ell+1}.
  \]
  Note that the partition $V = V_0 \cup V_1$ where $V_0 = \varnothing$ is
  trivially $(c_1, \alpha_1)$-good. In the following we argue that if for some
  $\ell< c^{-1}$ a partition $V = V_0 \cup \dotsb \cup V_\ell$ is $(c_\ell,
  \alpha_\ell)$-good, but not $(c_\ell, \alpha_\ell, \gamma_\ell)$-perfect, then
  there exists a partition $V = W_0 \cup \dotsb \cup W_{\ell + 1}$ that is
  $(c_{\ell + 1}, \alpha_{\ell + 1})$-good. This is sufficient to complete the
  proof, which can be seen as follows. By repeatedly applying the statement
  above, we either obtain a $(c_\ell, \alpha_\ell, \gamma_\ell)$-perfect
  partition $V = V_0 \cup \dotsb \cup V_\ell$ for some $1 \leq \ell < c^{-1}$ or
  we obtain a $(c_\ell, \alpha_\ell)$-good partition $V = V_0 \cup \dotsb \cup
  V_\ell$ for $\ell = \ceil{c^{-1}}$. In the former case, we are done because
  $c_\ell = c + \alpha/2^{\ell-1}$, $\alpha_\ell \leq \alpha$, and $\gamma_\ell
  \geq \gamma$ for $\ell < c^{-1}$. The latter case results in a contradiction,
  as we now verify. By averaging, there is some $i \in [\ell]$ such that $|V_i|
  \leq n/\ell$. Since $\ell \geq c^{-1}$, we then have $|V_i|\leq cn$.
  Furthermore, \ref{L-min-deg} states that $\delta(G[V_i]) \geq (c +
  \alpha/2^{\ell - 1})np \geq (c + \alpha/2^{1/c})np$, which in particular
  implies
  \[
    2e(V_i) \geq |V_i| (c + \alpha/2^{1/c})np.
  \]
  On the other hand, since $G$ is $(p, \beta)$-sparse, we have $2e(V_i) = e(V_i,
  V_i) \leq p|V_i|^2 + \beta |V_i|$, which combines with the above to yield
  \[
    (c + \alpha/2^{1/c})np \leq p|V_i| + \beta.
  \]
  Since $|V_i|\leq cn$ and $\beta \leq \eta np$, this is a contradiction if
  $\eta < \alpha/2^{1/c}$.

  We now prove the statement mentioned above. Assume that the partition $V = V_0
  \cup \dotsb \cup V_\ell$ is $(c_\ell, \alpha_\ell)$-good but not $(c_\ell,
  \alpha_\ell, \gamma_\ell)$-perfect, for some $\ell < c^{-1}$. Then there is
  some $i \in [\ell]$ and a partition $V_i = X \cup Y$ such that
  \begin{equation}\label{eq:non-perfect-tuple}
    e(X, Y) < \gamma_\ell |X||Y|p.
  \end{equation}
  We further define partitions $X = V_X \cup W_X$ and $Y = V_Y \cup W_Y$ as
  follows. Set
  \begin{equation}\label{eq:w0}
    W_X^0 := \{ v \in X : \deg(v, Y) \geq \alpha n p/2^\ell \},
  \end{equation}
  and, for every $j \geq 1$, set
  \begin{equation}\label{eq:wj}
    W_X^j :=
      \begin{cases}
        W_X^{j - 1} \cup \{v\}, & \text{if there exists $v \in X \setminus
        W_X^{j - 1}$ with $\deg(v, W_X^{j - 1}) \geq \alpha np/2^\ell$},
        \\
        W_X^{j - 1}, & \text{otherwise}.
      \end{cases}
  \end{equation}
  Finally, define $W_X := \bigcup_{j \geq 0} W_X^j$ and $V_X := X \setminus
  W_X$. The partition $Y = V_Y \cup W_Y$ is defined analogously (with $X$
  replaced by $Y$).

  We now define the partition $V = W_0 \cup \dotsb \cup W_{\ell + 1}$ by setting
  $W_0 = V_0 \cup W_X \cup W_Y$ and $(W_1, \dotsc, W_{\ell + 1}) = (V_1, \dotsc,
  V_{i - 1}, V_X, V_Y, V_{i + 1}, \dotsc, V_\ell)$, i.e., we obtain the new
  partition of $V$ by replacing the part $V_i$ by the two parts $V_X$ and $V_Y$
  and adding $W_X \cup W_Y$ to $V_0$. It remains to check that this partition is
  $(c_{\ell + 1}, \alpha_{\ell + 1})$-good.

  Showing \ref{L-size-bad} essentially boils down to proving that $W_X$ and
  $W_Y$ are not too large. As a first step, it follows from
  \eqref{eq:non-perfect-tuple} and \eqref{eq:w0} that
  \[
    |W_X^0| \cdot \alpha np/2^\ell \leq e(X, Y) < \gamma_{\ell}|X||Y|p \leq
    \gamma_\ell |V_i|^2p/4,
  \]
  and thus $|W_X^0| \leq 2^{\ell - 2} \gamma_\ell |V_i|/\alpha \leq 2^{\ell
  - 2}\gamma_\ell n/\alpha$. Assume towards a contradiction that $|W_X| > \alpha
  n/2^{\ell + 2}$. Then there must exist some $j \geq 0$ such that $|W_X^j| =
  \ceil{\alpha n/2^{\ell + 2}}$. We remark that by the choice $\gamma_\ell =
  \alpha^2/2^{2\ell + 1}$, we have $|W_X^0| \leq 2^{\ell - 2}\gamma_\ell n/a =
  \alpha n/2^{\ell + 3} \leq |W_X^j|/2$. From the definition of $W_X^j$, we
  moreover see that every vertex in $W_X^j \setminus W_X^0$ adds at least
  $\alpha np/2^\ell$ edges to $e(W_X^j)$. Therefore, we have
  \begin{equation}\label{eq:wj-lower-bound}
    e(W_X^j) \geq \frac{\alpha}{2^\ell} np \cdot |W_X^j \setminus W_X^0| \geq
    \frac{\alpha}{2^{\ell + 1}} np \cdot |W_X^j|.
  \end{equation}
  On the other hand, since $G$ is $(p, \beta)$-sparse,
  \begin{equation}\label{eq:wj-upper-bound}
    e(W_X^j) \leq |W_X^j|^2 p/2 + \beta |W_X^j| = \ceil*{\frac{\alpha n}{2^{\ell
    + 2}}} (|W_X^j| p/2 + \beta).
  \end{equation}
  As $\beta \leq \eta np \leq \eta 2^{\ell + 2}|W_X^j|p/\alpha$, we see that for
  small enough $\eta$, equations \eqref{eq:wj-lower-bound} and
  \eqref{eq:wj-upper-bound} result in a contradiction. It follows that $|W_X|
  \leq \alpha n/2^{\ell + 2}$ and one can show analogously that $|W_Y|\leq
  \alpha n/2^{\ell + 2}$. In conclusion,
  \[
    |W_0| = |V_0| + |W_X| + |W_Y| \leq (\alpha_\ell + \alpha/2^{\ell + 1})n =
    \alpha_{\ell + 1} n,
  \]
  completing the proof of \ref{L-size-bad}.

  Lastly, we prove \ref{L-min-deg}. Since $V = V_0 \cup \dotsb \cup V_\ell$ is
  $(c_\ell, \alpha_\ell)$-good, we have $\delta(G[V_i]) \geq c_\ell np = (c +
  \alpha/2^{\ell - 1})np$. Observe that \eqref{eq:w0} and \eqref{eq:wj} imply
  that $V_X$ contains only vertices with fewer than $\alpha np/2^{\ell}$
  neighbours in $W_X \cup Y = V_i \setminus V_X$. Therefore,
  \[
    \delta(G[V_X]) \geq (c + \alpha/2^{\ell - 1})np - \alpha n p/2^\ell = (c +
    \alpha/2^\ell)np = c_{\ell + 1}np.
  \]
  Similarly, we prove $\delta(G[V_Y]) \geq c_{\ell + 1}np$, thus establishing
  \ref{L-min-deg}.
\end{proof}

We can now prove Lemma \ref{lem:partitioning-lemma} which we restate for
convenience.

\partitioning*

\begin{proof}
  We may assume without loss of generality that $\xi < \alpha$. Let $c' = c +
  \alpha - \xi$, $\alpha' = c^2\xi/4$, and choose $\eta$ to be small enough so
  that the following arguments hold. Then $G$ has minimum degree at least $(c' +
  \alpha')np$, and so we can apply Lemma~\ref{lem:perfect-partition} to $G$ to
  obtain, for some $\gamma' > 0$ and some integer $1 \leq \ell < c^{-1}$, a
  $(c', \alpha', \gamma')$-perfect partition $V(G) = V_0' \cup V_1' \cup \dotsb
  \cup V_\ell'$. In the following, we distribute the vertices in $V_0'$ over the
  other sets to obtain a partition $V(G) = V_1 \cup \dotsb \cup V_\ell$ as in
  the statement of the lemma (in particular, we aim to have $V_i' \subseteq V_i$
  for every $i \in [\ell]$).

  Let $m = |V_0'|$ and note that by \ref{L-size-bad}, we have $m \leq \alpha'
  n$. It then follows from the fact that $G$ has minimum degree at least $(c' +
  \alpha')np \geq |V_0'|p + c'np$ and Lemma~\ref{lem:edges-out-small-set} that
  there exists an ordering $w_1, \dotsc, w_m$ of the vertices of $V_0'$ such
  that for every $j \in [m]$, we have
  \begin{equation}\label{eq:min-degree}
    \deg(w_j, V \setminus \{w_j, \dotsc, w_m\}) \geq c'np - \beta \geq cnp,
  \end{equation}
  where the last inequality holds by choosing $\eta$ to be small enough. We
  process the vertices $w_1, \dotsc, w_m$ in this order, defining $\ell$ chains
  of subsets
  \[
    \varnothing = W_i^0 \subseteq W_i^1 \subseteq \cdots \subseteq W_i^m
    \subseteq V_0' \qquad \text{for $i \in [\ell]$}
  \]
  along the way. For this, we set $W_i^0 = \varnothing$ for every $i \in
  [\ell]$, and for every vertex $w_j$, we do the following:
  \begin{enumerate}[(1)]
    \item\label{i*} choose an arbitrary $i^\star \in [\ell]$ satisfying
      $\deg(w_j, V'_{i^\star} \cup W_{i^\star}^{j - 1}) \geq cnp/\ell$,
    \item\label{Wj} set $W_{i^\star}^j := W_{i^\star}^{j - 1} \cup \{w_j\}$ and
      $W_i^j := W_i^{j - 1}$ for all $i \neq i^\star$.
  \end{enumerate}
  Observe that by \eqref{eq:min-degree} there always exists at least one
  $i^\star \in [\ell]$ as in \ref{i*}. Lastly, we define $V_i := W_i^m \cup
  V_i'$ for every $i \in [\ell]$.

  It is easy to see that $V_1, \dotsc, V_\ell$ contain all vertices of $V_0'$
  and that we have $\delta(G[V_i]) \geq cnp/\ell \geq c^2np$ for every $i \in
  [\ell]$. Moreover, by \ref{L-min-deg}, all but at most $|V_0'| \leq \alpha'n$
  vertices $v$ in each set $V_i$ satisfy $\deg(v, V_i) \geq c'np = (c + \alpha -
  \xi)np$. Since $G$ is $(p, \beta)$-sparse and $\delta(G[V_i]) \geq c^2 np$,
  one easily derives $|V_i| \geq c^2n/2$. Thus, by our choice of constants
  $\alpha' n \leq \xi|V_i|$ and so $\delta_\xi(G[V_i]) \geq (c + \alpha -
  \xi)np$. We finish the proof by showing that each graph $G[V_i]$ is a $\gamma
  p$-expander, with $\gamma = \min{\{ c^2/4, \gamma'/4 \}}$.

  For this, fix some partition $V_i = X \cup Y$ into non-empty sets where,
  without loss of generality, we assume $|X| \leq |Y|$. If $|X| \leq c^2 n/2$,
  then it follows from $\delta(G[V_i]) \geq c^2np \geq |X|p + c^2np/2$ and
  Lemma~\ref{lem:edges-out-small-set} that $e(X, Y) \geq c^2|X|np/2 - \beta|X|
  \geq \gamma |X||Y|p$, for $\eta$ small enough. On the other hand, if $|X|, |Y|
  \geq c^2n/2$, then we use the fact that $G[V_i']$ is a $\gamma' p$-expander to
  get
  \[
    e(X, Y) \geq e(X \cap V_i', Y \cap V_i') \geq \gamma'|X \cap V_i'| |Y \cap
    V_i'|p.
  \]
  The assumption on the sizes of $X$ and $Y$ implies $|X \cap V_i'| = |X
  \setminus V'_0| \geq |X| - \alpha' n \geq |X|/2$ and similarly $|Y \cap V_i'|
  \geq |Y|/2$. This gives
  \[
    e(X, Y) \geq \gamma'|X||Y|p/4 \geq \gamma |X||Y|p,
  \]
  as required.
\end{proof}

\section{Embedding and boosting}\label{sec:embedding-lemma}

In this section we give the proof of the Embedding Lemma
(Lemma~\ref{lem:embedding-lemma}). The proof relies on the following approximate
version covering almost all the vertices of $G$.

\begin{lemma}\label{lem:embedding-cheat}
  For every integer $k \geq 2$ and all $\alpha, \mu > 0$, there exists a
  positive $\eta(\alpha, k)$ such that the following holds for all sufficiently
  large $n$. Let $p \in (0, 1)$ and $\beta \leq \eta np$ and let $G$ be a $(p,
  \beta)$-sparse graph on $n$ vertices such that
  \[
    \delta(G) \geq 2\alpha np \quad \text{and} \quad \delta_{\alpha/4}(G) \geq
    (1/k + \alpha)np.
  \]
  Then $G$ contains a collection of $k - 1$ cycles covering all but at most $\mu
  n$ vertices.
\end{lemma}
\begin{proof}[Proof sketch]
  Since the argument is fairly standard nowadays, we only give a rough sketch of
  the proof. We apply the sparse regularity lemma
  (Lemma~\ref{lem:sparse-regularity-lemma}) to the graph $G$ with a sufficiently
  small parameter $\eps > 0$. Let $t$ be the number of vertices in the reduced
  graph $R$. Since $G$ is $(p, \beta)$-sparse, straightforward calculations show
  that $\delta(R) \geq \alpha t$ and $\delta_{\alpha/4+\eps}(R) \geq
  (1/k+\alpha-2\eps)t$. Let $U$ be the set containing the at most
  $(\alpha/4+\eps)t$ vertices (clusters) $v \in V(R)$ with $\deg_R(v) <
  (1/k+\alpha-2\eps)t$. A simple greedy strategy allows us to find a matching
  $M$ in $R$ that saturates the set $U$; this matching contains at most
  $2|U|\leq (\alpha/2+2\eps)t$ vertices. Let $W = V(R) \setminus V(M)$. Since
  $R[W]$ has minimum degree at least $(1/k+\alpha-2\eps)t - 2|U| \geq t/k$, it
  can be covered by at most $k - 1$ cycles, by Theorem~\ref{thm:KL}. Moreover,
  the minimum degree of $R$ ensures that we can select a different neighbour in
  $W$ for each of the vertices in $V(M)$.

  Using standard machinery, one can now translate each cycle in the covering of
  $R[W]$, as well as all the matching edges in $M$ that have an endpoint whose
  selected neighbour lies on that cycle, into a single cycle that covers all but
  at most $O(\eps n)$ vertices of the graph $G$. The following figure
  schematically represents one such cycle together with the two edges of $M$
  that have an endpoint whose selected neighbour lies on the cycle.

  
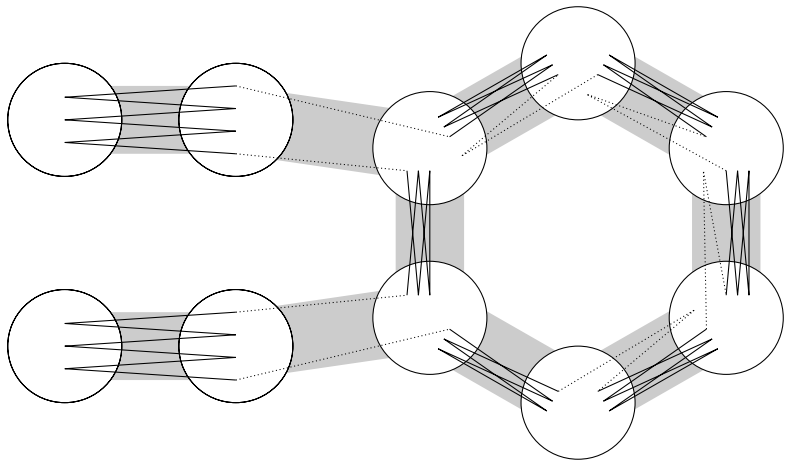
\begin{center}
    \begin{tikzpicture}[scale=1.5]
      \coordinate (Ma) at (0.5,1);
      \coordinate (Mb) at (2,1);
      \coordinate (Mu) at (0.5,-1);
      \coordinate (Mv) at (2,-1);
      \draw[line width=0.9cm,color=black!20] (Ma)--(Mb) (Mu)--(Mv);
      \begin{scope}[xshift=5cm]
        \foreach \i in {1,...,6} {
          \coordinate (C\i) at (60*\i-30:1.5cm);
        }
      \end{scope}
      \draw[line width=0.9cm,color=black!20] (Mb)--(C3) (Mv)--(C4);
      \draw[line width=0.9cm,color=black!20] (C6)--(C1) (C1)--(C2) (C2)--(C3)
      (C3)--(C4) (C4)--(C5) (C5)--(C6);
      \foreach \i in {a,b,u,v} {
        \draw[fill=white] (M\i) circle (0.5cm) circle (0.5cm);
      }
      \foreach \i in {1,...,6} {
        \draw[fill=white] (C\i) circle (0.5cm);
      }

      \foreach \i in {1,2,4,5,6} {
        \begin{scope}[xshift=5cm,rotate=60*\i]
          \draw (30:1.5cm) ++(0,-0.2) coordinate (f\i) ++(0.1,0) coordinate (au1)
          ++(0.1,0) coordinate (au2);
          \draw (30:1.5cm) +(-0.2,-0.2) coordinate (h\i);
          \draw (-30:1.5cm) ++(0,0.2) coordinate (g\i) ++(0.1,0) coordinate (au3)
          ++(0.1,0) coordinate (au4);
          \draw[black]
          (f\i) -- (au3) -- (au2) -- (au4) -- (au1) -- (g\i);
        \end{scope}
      }
      \begin{scope}[xshift=5cm,rotate=60*3]
        \draw (30:1.5cm) ++(0,-0.2) coordinate (au2) ++(0.1,0) coordinate (au1)
        ++(0.1,0) coordinate (f3);
        \draw (30:1.5cm) +(-0.2,-0.2) coordinate (h3);
        \draw (-30:1.5cm) ++(0,0.2) coordinate (au4) ++(0.1,0) coordinate (au3)
        ++(0.1,0) coordinate (g3);
        \draw[black]
        (f3) -- (au3) -- (au2) -- (au4) -- (au1) -- (g3);
      \end{scope}
      \draw[black] (Mb) ++(0,-0.3) coordinate (a) {} -- ++(-1.5,0.1) --
      +(1.5,0.1) -- ++(0,0.2) -- +(1.5,0.1) -- ++(0,0.2) -- +(1.5,0.1)
      coordinate (b) {};
      \draw[black] (Mv) ++(0,-0.3) coordinate (u) {} -- ++(-1.5,0.1) --
      +(1.5,0.1) -- ++(0,0.2) -- +(1.5,0.1) -- ++(0,0.2) -- +(1.5,0.1)
      coordinate (v) {};

      \draw[densely dotted] (b)--(f2) (a)--(g3) (v)--(f3) (u)--(g4);
      \draw[densely dotted] (f1)--(h2)--(g2) (f4)--(h5)--(g5)
      (f5)--(h6)--(g6) (f6)--(h1)--(g1);
    \end{tikzpicture}
  \end{center}

  The long black paths---each of which covers all but $O(\eps n/t)$ vertices in
  both of the clusters it belongs to---are embedded first into the dense regular
  pairs, using, e.g., the method from the proof
  of~\cite[Lemma~2.3]{ben2012long}. Finally, using the definition of an $(\eps,
  p)$-regular pair, the dotted edges can be added by sacrificing at most $O(\eps
  n/t)$ vertices from each black path. In this way, one obtains a collection of
  $k-1$ cycles in $G$ covering all but $O(\eps n) \leq \mu n$ vertices.
\end{proof}

Utilising a trick of Nenadov and the second author~\cite{nenadov2020komlos}, as
a consequence of Lemma~\ref{lem:embedding-cheat} we get
Lemma~\ref{lem:embedding-lemma}.

\embedding*

\begin{proof}
  Let $C = C(\alpha, k) > 0$ be a sufficiently large constant for the arguments
  below to go through, $\eta = \eta_{\ref{lem:embedding-cheat}}(\alpha/4, k)$,
  and $\mu = \min{\{1/2, \alpha/48, \eta/8, 1/(4C) \}}$.

  Take $m$ to be the largest integer such that $n/2^{m - 1} > \max{\{
  \floor{\eta^{-1}\beta/p}, \floor{C\log n/p} \}}$ and note that $m \leq \log_2
  n$. We first claim that there exists a partition $V(G) = L_1 \cup \dotsb \cup
  L_m$ such that:
  \begin{enumerate}[(i), font=\itshape]
    \item\label{boost1} $|L_i| = \floor{n/2^i}$ for all $i \in [m - 1]$,
    \item\label{boost2} $\deg(v, L_i) \geq \alpha|L_i|p$ for all $v \in V(G)$
      and $i \in [m]$, and
    \item\label{boost3} for every $i, j \in [m]$, we have $\deg(v, L_i) \geq
      (1/k + \alpha/2)|L_i|p$ for all but at most $(\alpha/24)|L_j|$ vertices $v
      \in L_j$.
  \end{enumerate}
  Indeed, a partition $V(G) = L_1 \cup \dotsb \cup L_m$ chosen uniformly at
  random among all partitions such that $|L_i| = \floor{n/2^i}$ for all $i \in
  [m - 1]$ has these properties with high probability. We briefly explain how
  one can conclude this. Observe that
  \begin{equation}\label{eq:leftover-lb}
    |L_m| = n - \sum_{i = 1}^{m - 1} \floor{n/2^i} \geq n - n \sum_{i = 1}^{m -
    1} 1/2^i = n/2^{m - 1},
  \end{equation}
  and thus $|L_i| \geq C\log n/p$ for all $i \in [m]$. For every fixed $v \in
  V(G)$ and $i \in [m]$, the random variable $\deg(v, L_i)$ follows a
  hypergeometric distribution with mean at least $2\alpha|L_i|p \geq 2\alpha
  C\log n$, so by Chernoff's inequality, we have $\Pr[\deg(v, L_i) \leq
  \alpha|L_i|p] \leq n^{-2}$. Thus, the union bound shows that $\ref{boost2}$
  holds with high probability for every $i \in [m]$ and $v \in V(G)$. The
  statement in $\ref{boost3}$ follows similarly.

  Fix a partition $V(G) = L_1 \cup \dotsb \cup L_m$ satisfying
  $\ref{boost1}$--$\ref{boost3}$. We show by induction that for every $i \in
  [m]$, the graph $G[L_1 \cup \dotsb \cup L_i]$ contains $k - 1$ path forests
  with at most $i$ paths each, which together cover all but at most $\mu|L_i|$
  vertices in $L_1 \cup \dotsb \cup L_i$. By maximality of $m$, we have
  \[
    |L_m| \leq n/2^{m - 1} + m \leq 2\max{\{ \eta^{-1} \beta/p, C\log n/p \}} +
    \log_2 n \leq 4\max{\{ \eta^{-1} \beta/p, C\log n/p \}}.
  \]
  Hence, for $i = m$, adding at most $\mu|L_m|$ uncovered vertices to one of the
  path forests used to cover $G[L_1 \cup \dotsb \cup L_m]$ results in a cover of
  $G$ by $k - 1$ path forests, each of which contains at most $(k - 1)m + \mu
  \cdot 4\max{\{ \eta^{-1}\beta/p, C\log n/p \}} \leq \max{\{ \beta/p, \log n/p
  \}}$ paths, for our choice of $\mu$.

  For the base case $i = 1$, this is a consequence of
  Lemma~\ref{lem:embedding-cheat} applied with $\alpha/2$ (as $\alpha$) and
  $\mu$ to $G[L_1]$. We can indeed do so since $\beta \leq \eta|L_1|p$.

  Assume now the statement holds for some $1 \leq i < m$ and let us show it for
  $i + 1$. Denote by $W$ the vertices not covered by the $k - 1$ path forests in
  $G[L_1 \cup \dotsb \cup L_i]$ and note that $|W| \leq \mu|L_i| \leq 2\mu|L_{i
  + 1}|$. For every $v \in V(G)$ we have
  \[
    \deg_G(v, L_{i+1} \cup W) \geq \deg_G(v, L_{i+1}) \osref{$\ref{boost2}$}\geq
    \alpha|L_{i+1}|p \geq \alpha\frac{|L_{i+1}| + |W|}{1 + 4\mu} p \geq
    (\alpha/2)|L_{i+1} \cup W|p.
  \]
  Similar calculation shows that every vertex $v \in V(G)$ with $\deg_G(v,
  L_{i+1}) \geq (1/k + \alpha/2)|L_{i+1}|p$ also satisfies $\deg_G(v, L_{i+1}
  \cup W) \geq (1/k + \alpha/4)|L_{i+1} \cup W|p$. Since there were at most
  $(\alpha/24)|L_{i+1}|$ vertices in $L_{i+1}$ violating the previous degree
  condition by $\ref{boost3}$, and $|W| \leq \mu|L_{i+1}|$, there are at most
  $(\alpha/16)|L_{i+1}|$ vertices $v \in L_{i+1} \cup W$ with $\deg_G(v, L_{i+1}
  \cup W) < (1/k + \alpha/4)|L_{i+1} \cup W|p$. Therefore, another application
  of Lemma~\ref{lem:embedding-cheat} for $\alpha/4$ (as $\alpha$) to $G[L_{i+1}
  \cup W]$ shows that the hypothesis holds for $i+1$. Once again, we may apply
  the lemma as $\beta \leq \eta|L_{i+1} \cup W|p$ by \eqref{eq:leftover-lb}.
\end{proof}

\subsection*{Acknowledgements.} The authors would like to thank the anonymous
reviewers for their thorough reading of the paper and valuable comments; in
particular, for pointing out a slight gap in a previous version of the paper.

\bibliographystyle{abbrv} \bibliography{references}

\begin{thebibliography}{10}

\bibitem{allen2020bandwidth}
P.~Allen, J.~B{\"o}ttcher, J.~Ehrenm{\"u}ller, and A.~Taraz.
\newblock The bandwidth theorem in sparse graphs.
\newblock {\em Advances in Combinatorics}, 2020(6):1--60, 2020.

\bibitem{balogh2011local}
J.~Balogh, B.~Csaba, and W.~Samotij.
\newblock Local resilience of almost spanning trees in random graphs.
\newblock {\em Random Structures \& Algorithms}, 38(1-2):121--139, 2011.

\bibitem{balogh2012corradi}
J.~Balogh, C.~Lee, and W.~Samotij.
\newblock Corr{\'a}di and {H}ajnal's theorem for sparse random graphs.
\newblock {\em Combinatorics, Probability and Computing}, 21(1-2):23--55, 2012.

\bibitem{balogh2017stability}
J.~Balogh, F.~Mousset, and J.~Skokan.
\newblock Stability for vertex cycle covers.
\newblock {\em The Electronic Journal of Combinatorics}, 24(3):P3--56, 2017.

\bibitem{ben2012long}
I.~Ben-Eliezer, M.~Krivelevich, and B.~Sudakov.
\newblock Long cycles in subgraphs of (pseudo) random directed graphs.
\newblock {\em Journal of Graph Theory}, 70(3):284--296, 2012.

\bibitem{bottcher2017large}
J.~B{\"o}ttcher.
\newblock Large-scale structures in random graphs.
\newblock {\em Surveys in Combinatorics 2017}, 440:87, 2017.

\bibitem{bottcher2009proof}
J.~B{\"o}ttcher, M.~Schacht, and A.~Taraz.
\newblock Proof of the bandwidth conjecture of {B}ollob{\'a}s and {K}oml{\'o}s.
\newblock {\em Mathematische Annalen}, 343(1):175--205, 2009.

\bibitem{chau2011posa}
P.~Ch{\^a}u, L.~DeBiasio, and H.~A. Kierstead.
\newblock P{\'o}sa's conjecture for graphs of order at least $2 \times 10^8$.
\newblock {\em Random Structures \& Algorithms}, 39(4):507--525, 2011.

\bibitem{conlon2014combinatorial}
D.~Conlon.
\newblock Combinatorial theorems relative to a random set.
\newblock In {\em Proceedings of the {I}nternational {C}ongress of
  {M}athematicians---{S}eoul 2014. {V}ol. {IV}}, pages 303--327. Kyung Moon Sa,
  Seoul, 2014.

\bibitem{dirac1952some}
G.~A. Dirac.
\newblock Some theorems on abstract graphs.
\newblock {\em Proceedings of the London Mathematical Society}, 3(1):69--81,
  1952.

\bibitem{enomoto1987p_3}
H.~Enomoto, A.~Kaneko, and Z.~Tuza.
\newblock {$P_3$}-factors and covering cycles in graphs of minimum degree
  {$n/3$}.
\newblock In {\em Combinatorics ({E}ger, 1987)}, volume~52 of {\em Colloq.
  Math. Soc. J\'{a}nos Bolyai}, pages 213--220. North-Holland, Amsterdam, 1988.

\bibitem{erdos1964problem}
P.~Erd{\H{o}}s.
\newblock {Problem 9, Theory of graphs and its applications (M. Fieldler ed.)}.
\newblock {\em Czech. Acad. Sci. Publ., Prague}, pages 159--159, 1964.

\bibitem{erdHos1991vertex}
P.~Erd{\H{o}}s, A.~Gy{\'a}rf{\'a}s, and L.~Pyber.
\newblock Vertex coverings by monochromatic cycles and trees.
\newblock {\em Journal of Combinatorial Theory, Series B}, 51(1):90--95, 1991.

\bibitem{ferber2017robust}
A.~Ferber, R.~Nenadov, A.~Noever, U.~Peter, and N.~{\v{S}}kori{\'c}.
\newblock Robust {H}amiltonicity of random directed graphs.
\newblock {\em Journal of Combinatorial Theory, Series B}, 126:1--23, 2017.

\bibitem{georgakopoulos2018spanning}
A.~Georgakopoulos, J.~Haslegrave, R.~Montgomery, and B.~Narayanan.
\newblock Spanning surfaces in 3-graphs.
\newblock {\em arXiv preprint arXiv:1808.06864}, 2018.

\bibitem{gerke2007small}
S.~Gerke, Y.~Kohayakawa, V.~R{\"o}dl, and A.~Steger.
\newblock Small subsets inherit sparse $\varepsilon$-regularity.
\newblock {\em Journal of Combinatorial Theory, Series B}, 97(1):34--56, 2007.

\bibitem{glock2021decompositions}
S.~Glock, D.~K{\"u}hn, R.~Montgomery, and D.~Osthus.
\newblock Decompositions into isomorphic rainbow spanning trees.
\newblock {\em Journal of Combinatorial Theory, Series B}, 146:439--484, 2021.

\bibitem{hajnal1970proof}
A.~Hajnal and E.~Szemer{\'e}di.
\newblock Proof of a conjecture of {P}.\ {E}rd{\H{o}}s.
\newblock {\em Combinatorial theory and its applications}, 2:601--623, 1970.

\bibitem{haxell1995condition}
P.~E. Haxell.
\newblock A condition for matchability in hypergraphs.
\newblock {\em Graphs and Combinatorics}, 11(3):245--248, 1995.

\bibitem{komlos1998proof}
J.~Koml{\'o}s, G.~N. S{\'a}rk{\"o}zy, and E.~Szemer{\'e}di.
\newblock Proof of the {S}eymour conjecture for large graphs.
\newblock {\em Annals of Combinatorics}, 2(1):43--60, 1998.

\bibitem{kouider1996covering}
M.~Kouider and Z.~Lonc.
\newblock Covering cycles and $k$-term degree sums.
\newblock {\em Combinatorica}, 16(3):407--412, 1996.

\bibitem{krivelevich1997triangle}
M.~Krivelevich.
\newblock Triangle factors in random graphs.
\newblock {\em Combinatorics, Probability and Computing}, 6(3):337--347, 1997.

\bibitem{krivelevich2006pseudo}
M.~Krivelevich and B.~Sudakov.
\newblock Pseudo-random graphs.
\newblock In {\em More sets, graphs and numbers}, pages 199--262. Springer,
  2006.

\bibitem{kwan2020almost}
M.~Kwan.
\newblock {Almost all Steiner triple systems have perfect matchings}.
\newblock {\em Proceedings of the London Mathematical Society},
  121(6):1468--1495, 2020.

\bibitem{lee2012dirac}
C.~Lee and B.~Sudakov.
\newblock Dirac's theorem for random graphs.
\newblock {\em Random Structures \& Algorithms}, 41(3):293--305, 2012.

\bibitem{levitt2010avoid}
I.~Levitt, G.~N. S{\'a}rk{\"o}zy, and E.~Szemer{\'e}di.
\newblock How to avoid using the regularity lemma: {P}{\'o}sa’s conjecture
  revisited.
\newblock {\em Discrete Mathematics}, 310(3):630--641, 2010.

\bibitem{montgomery2019hamiltonicity}
R.~Montgomery.
\newblock Hamiltonicity in random graphs is born resilient.
\newblock {\em Journal of Combinatorial Theory, Series B}, 139:316--341, 2019.

\bibitem{montgomery2019spanning}
R.~Montgomery.
\newblock Spanning trees in random graphs.
\newblock {\em Advances in Mathematics}, 356:106793, 2019.

\bibitem{montgomery2020hamiltonicity}
R.~Montgomery.
\newblock Hamiltonicity in random directed graphs is born resilient.
\newblock {\em Combinatorics, Probability and Computing}, 29(6):900--942, 2020.

\bibitem{nenadov2019powers}
R.~Nenadov and N.~{\v{S}}kori{\'c}.
\newblock Powers of {H}amilton cycles in random graphs and tight {H}amilton
  cycles in random hypergraphs.
\newblock {\em Random Structures \& Algorithms}, 54(1):187--208, 2019.

\bibitem{nenadov2020komlos}
R.~Nenadov and N.~{\v{S}}kori{\'c}.
\newblock {On Koml{\'o}s’ tiling theorem in random graphs}.
\newblock {\em Combinatorics, Probability and Computing}, 29(1):113--127, 2020.

\bibitem{nenadov2019resilience}
R.~Nenadov, A.~Steger, and M.~Truji{\'c}.
\newblock Resilience of perfect matchings and {H}amiltonicity in random graph
  processes.
\newblock {\em Random Structures \& Algorithms}, 54(4):797--819, 2019.

\bibitem{rodl2006dirac}
V.~R{\"o}dl, A.~Ruci{\'n}ski, and E.~Szemer{\'e}di.
\newblock A {D}irac-type theorem for 3-uniform hypergraphs.
\newblock {\em Combinatorics, Probability and Computing}, 15(1-2):229--251,
  2006.

\bibitem{scott2011szemeredi}
A.~Scott.
\newblock Szemer{\'e}di's regularity lemma for matrices and sparse graphs.
\newblock {\em Combinatorics, Probability and Computing}, 20(3):455--466, 2011.

\bibitem{seymour1973problem}
P.~Seymour.
\newblock Problem section.
\newblock In {\em Combinatorics: Proceedings of the {B}ritish {C}ombinatorial
  {C}onference}, pages 201--202, 1973.

\bibitem{vskoric2018local}
N.~{\v{S}}kori{\'c}, A.~Steger, and M.~Truji{\'c}.
\newblock Local resilience of an almost spanning $k$-cycle in random graphs.
\newblock {\em Random Structures \& Algorithms}, 53(4):728--751, 2018.

\bibitem{sudakov2017robustness}
B.~{Sudakov}.
\newblock {Robustness of graph properties}.
\newblock In {\em Surveys in combinatorics 2017}, volume 440 of {\em London
  Math. Soc. Lecture Note Ser.}, pages 372--408. Cambridge Univ. Press,
  Cambridge, 2017.

\bibitem{sudakov2008local}
B.~Sudakov and V.~H. Vu.
\newblock Local resilience of graphs.
\newblock {\em Random Structures \& Algorithms}, 33(4):409--433, 2008.

\bibitem{thomason1987pseudo}
A.~{Thomason}.
\newblock {Pseudo-random graphs}.
\newblock In {\em Random graphs '85 ({P}ozna\'{n}, 1985)}, volume 144 of {\em
  North-Holland Math. Stud.}, pages 307--331. North-Holland, Amsterdam, 1987.

\bibitem{thomason1987random}
A.~Thomason.
\newblock Random graphs, strongly regular graphs and pseudorandom graphs.
\newblock {\em Surveys in Combinatorics}, 123:173--195, 1987.

\end{thebibliography}

\end{document}